\documentclass[reqno,a4paper,12pt]{amsart}

\usepackage[all,poly]{xy}
\usepackage{amsfonts}
\usepackage[mathcal]{eucal}
\usepackage{eufrak}
\usepackage{amssymb}
\usepackage{amsmath}
\usepackage{mathrsfs}
\usepackage{color}
\usepackage[colorlinks]{hyperref}
\usepackage{enumerate}

\definecolor{citecol}{RGB}{145, 1, 1}

\hypersetup{colorlinks=true,citecolor=citecol,linkcolor=blue,linktocpage=true}

\theoremstyle{plain}
\newtheorem {lemma}{Lemma}[section] 
\newtheorem {theorem}[lemma]{Theorem}
\newtheorem {fact}[lemma]{Fact}
\newtheorem {thm}[lemma]{Theorem}

\newtheorem {cor}[lemma]{Corollary}

\newtheorem {prop}[lemma]{Proposition}

\theoremstyle{definition}
\newtheorem {remark}[lemma]{Remark}

\newtheorem {example}[lemma]{Example}

\theoremstyle{definition}

\newtheorem{deff}[lemma]{Definition}{}

\newcommand{\M}{\operatorname{\mathbb M}}

\newcommand{\reg}{\operatorname{reg}}
\newcommand{\ssink}{\operatorname{sink}}

\newcommand{\K}{\mathsf{k}}

\def\M{\mathbb{M}}

\newcommand{\SP}{\operatorname{SP}}

\begin{document}

\title[Sandpile monoids and weighted Leavitt path algebras]{ON STRUCTURAL CONNECTIONS BETWEEN SANDPILE MONOIDS
	AND WEIGHTED LEAVITT PATH ALGEBRAS}


\author{Roozbeh Hazrat}\address{
Centre for Research in Mathematics and Data Science\\
Western Sydney University\\
Australia}
\email{r.hazrat@westernsydney.edu.au}

\author{Tran Giang  Nam}
\address{Institute of Mathematics, VAST, 18 Hoang Quoc Viet, Cau Giay, Hanoi, Vietnam}
\email{tgnam@math.ac.vn}



\begin{abstract}
In this article, we establish the relations between a sandpile graph, its sandpile monoid and the weighted Leavitt path algebra associated with it. Namely, we show that the lattice of all idempotents of the sandpile monoid $\text{SP}(E)$ of a sandpile graph $E$ is both isomorphic to the lattice of all nonempty saturated hereditary subsets of $E$, the lattice of all order-ideals of $\text{SP}(E)$ and the lattice of all ideals of the weighted Leavitt path algebra $L_{\K}(E, \omega)$ generated by vertices. Also, we describe the sandpile group of a sandpile graph $E$ via archimedean classes of $\text{SP}(E)$, and prove that all maximal subgroups of $\text{SP}(E)$ are exactly the Grothendieck groups of these archimedean classes. Finally, we give the structure of the Leavitt path algebra $L_{\K}(E)$ of a sandpile graph $E$ via a finite chain of graded ideals being invariant under every graded automorphism of $L_{\K}(E)$, and completely describe the structure of $L_{\K}(E)$ such that the lattice of all idempotents of $\text{SP}(E)$ is a chain. Consequently, we completely describe the structure of the weighted Leavitt path algebra of a sandpile graph $E$ such that $\text{SP}(E)$ has exactly two idempotents. \medskip

\medskip

\end{abstract}


\maketitle

\section{Introduction} \label{introkai}

The notion of abelian sandpile models was invented in $1987$ by Bak, Tang and Wiesenfeld~\cite{bak}. While defined by a simple local rule, it produces self-similar global patterns that call for an explanation.  The models have been used to describe phenomena such as forest fires, traffic jams, stock market fluctuations, etc. The book of Bak~\cite{bakbook} describes how  events in nature apparently follow this type of behaviour. In~\cite{D1} Dhar systematically associated finite commutative monoid and groups to  sandpile models, now called respectively \emph{sandpile monoids} and \emph{sandpile groups}, and championed the use of them as an invariant which proved to capture many properties of the model. These algebraic structures constitute one of the main themes of the subject and they have been extensively studied (see, e.g., \cite{T, BT,CGGMS} and references there). The abelian sandpile model was independently discovered in $1991$ by Bj\"{o}rner, Lov\'{a}sz and Shor \cite{BLS}, who called it {\it chip-firing}. Indeed, in the last two decades the subject has been enriched by an exhilarating interaction of numerous areas of mathematics, including statistical physics, combinatorics, free boundary PDE, probability, potential theory, number theory and group theory; refer to Levine and Peres's recent survey on the subject \cite{LevinePeres} in more details.

In a different realm, the notion of Leavitt path algebras $L_{\K}(E)$ associated to directed graphs $E$ were introduced by Abrams and Aranda Pino in \cite{aap05}, and independently Ara, Moreno, and Pardo in \cite{amp}. These are a generalization of algebras introduced by William Leavitt in  1962~\cite{vitt62} as a ``universal'' algebra  $A$ of type $(1,n)$, so that  $A\cong A^{n}$ as right $A$-modules, where $n\in \mathbb N^+$. It was established quite early on in the theory that Leavitt path algebras only produce rings of ranks $(1,n)$. In \cite{vitt62} Leavitt further constructed rings of type $(m,n)$, where $1<m<n$ and showed that these algebras surprisingly are domains.  When $m\ge 2$, this universal ring is not realizable as a Leavitt path algebra. With this in mind, the notion of {\it weighted}  Leavitt path algebras $L_{\K}(E, w)$ associated to weighted graphs $(E, w)$  were introduced by the first author in \cite{H} (see \cite{Pre} for a nice overview of this topic).  The weighted Leavitt path algebras provide a natural context in which all of Leavitt's algebras (corresponding to any pair $m, n \in \mathbb{N}$) can be realized as a specific example.

The study of the commutative monoid $\mathcal{V}(A)$ of isomorphism classes of finitely generated projective left modules of a unital ring $A$ (with operation $\oplus$) goes back to the work of Grothendieck and Serre.  For a Leavitt path algebra $L_{\K}(E)$, the monoid $\mathcal{V}(L_{\K}(E))$ has received substantial attention since the introduction of the topic. Furthermore, the monoid $\mathcal{V}(L_{\K}(E, w))$ has been described completely by the first author in \cite{H}, and subsequently by Preusser in \cite{P}.  
In~\cite{GeneRooz} Abrams and the first author established that these monoids could be naturally related to the sandpile monoids of graphs. This relationship allows us to associate weighted Leavitt path algebras to the theory of sandpile models, thereby opening up an avenue by which to investigate sandpile models via the structure of weighted Leavitt path algebras, and vice versa. 

In this article, we investigate the relations between a sandpile graph, its sandpile monoid and the weighted Leavitt path algebra associated with it.  In particular we concentrate on order-ideals and idempotents of a sandpile monoid and establish a relation between them and certain ideals and the structure of weighted Leavitt path algebras.

 In Section \ref{weightedgraphssec2}, we describe the structure of sandpile monoid $\text{SP}(E)$ of a sandpile graph $E$ in terms of its sandpile submonoids $\text{SP}(E_H)$, where $E_H$ is the restriction graph of $E$ to an nonempty saturated hereditary subset $H$ of $E$, and show that the lattice of all order-ideals of $\text{SP}(E)$ is isomorphic to the lattice of all nonempty saturated hereditary subsets of $E$ (Theorem \ref{sandmonstruc}). Consequently, we obtain that every isomorphism between two sandpile monoids $\text{SP}(E)$ and $\text{SP}(F)$ induces an isomorphism between the lattices of all nonempty saturated hereditary subsets of $E$ and $F$ (Corollary \ref{sat-her-lat-iso}).

In Section \ref{weightedgraphssec3}, based on Babai and Toumpakari's result \cite[Theorem 4.13]{BT} and Theorem \ref{sandmonstruc}, we describe all idempotents and all order-ideals of the sandpile monoid $\text{SP}(E)$ of a sandpile graph $E$ in terms of the combinatorics of $E$. In particular, we obtain that they are not only isomorphic to each other as lattices, but also isomorphic to the lattice of all nonempty saturated hereditary subsets of $E$ (Theorem \ref{idem-fiter-here-ordideal}). Also,  we prove that a sandpile group $G(E)$ is the direct limit of the Grothendieck groups of  archimedean classes of $\text{SP}(E)$ and obtain that the Grothendieck group of an archimedean class $[x]$ is exactly the sandpile group of the sandpile graph arising from the support of $x$, as well as show that all maximal subgroups of $\text{SP}(E)$ are exactly the Grothendieck groups of  archimedean classes of $\text{SP}(E)$ (Theorem \ref{arch-theo}).

In Section \ref{weightedgraphssec4}, we make a bridge between the structure of sandpile graphs and the algebraic structure of Leavitt path algebras. Based on Theorem \ref{idem-fiter-here-ordideal},
we show that the lattice of all idempotents of the sandpile momoid of a sandpile graph $E$ is both isomorphic to the lattice of all nonzero graded ideals of the Leavitt path algebra $L_{\K}(E)$ and the lattice of all ideals the weighted Leavitt path algebra $L_{\K}(E, \omega)$ generated by vertices (Theorem \ref{classideal}). Also, we give the structure of the Leavitt path algebra $L_{\K}(E)$ of a sandpile graph $E$ via a finite chain of graded ideals being invariant under every graded automorphism of $L_{\K}(E)$, which extends Zelmanov et al.'s result \cite[Theorem 1]{Zel} to the sandpile graph setting, and completely describe the structure of $L_{\K}(E)$ such that the lattice of all idempotents of $\text{SP}(E)$ is a chain (Theorem \ref{structheo1}). Consequently, we completely describe the structure of the (weighted) Leavitt path algebra of a sandpile graph $E$ such that  $\text{SP}(E)$  has exactly two idempotents (Theorem \ref{twoidem}). 

Throughout we write $\mathbb N$ for the set of non-negative integers, and $\mathbb N^+$ for the set of positive integers.

\section{Sandpile monoids and groups} \label{weightedgraphssec2}
The main aim of this section is to describe the structure of sandpile monoid $\text{SP}(E)$ of a sandpile graph $E$ in terms of its sandpile submonoids $\text{SP}(E_H)$, where $E_H$ is the restriction graph of $E$ to an nonempty saturated hereditary subset $H$ of $E$, as well as to show that the lattice of all order-ideals of $\text{SP}(E)$ is isomorphic to the lattice of all nonempty saturated hereditary subsets of $E$ (Theorem \ref{sandmonstruc}). Consequently, we obtain  that an isomorphism between two sandpile monoids $\text{SP}(E)$ and $\text{SP}(F)$ yields an isomorphism between the lattices of all nonempty saturated hereditary subsets of $E$ and $F$ (Corollary \ref{sat-her-lat-iso}).\medskip

In order to define sandpile monoids  we need to recall the notion of a graph. 

A \emph{directed graph} is a quadruple $E=(E^0,E^1,s,r)$, where $E^0$ and $E^1$ are sets and $s,r:E^1\rightarrow E^0$ are maps.   Throughout, ``graph" will always mean ``directed graph".  The elements of $E^0$ are called 
\emph{vertices} and the elements of $E^1$ \emph{edges}.   (We allow the empty set to be viewed as a graph with $E^0 = E^1 = \emptyset.$)   If $e$ is an edge, then $s(e)$ is called its \emph{source} and $r(e)$ its \emph{range}. If $v$ is a vertex and $e$ an edge, we say that $v$ {\it emits} $e$ if $s(e)=v$, and $v$ {\it receives} $e$ if $r(e)=v$.   An edge $e$ is called a {\it loop at} $v$ in case $s(e) = v = r(e)$.  
A vertex is called a
{\it sink} if it emits no edges, and is called \emph{irrelevant} in case it emits exactly one edge.  A graph is called {\it reduced} if it contains no irrelevant vertices.  
A vertex is called \emph{regular} if it is not a sink and does not emit infinitely many edges.  The subset of $E^0$ consisting of all the regular vertices is denoted by $E^0_{\reg}$. Similarly, the subset of $E^0$ consisting of all the sinks  is denoted by $E^0_{\ssink}$.

The \emph{out-degree} of a vertex $v$, denoted by $\text{out-deg}(v)$, is defined as $|s^{-1}(v)|$. A graph is called \emph{row-finite} if any vertex emits a finite number (possibly zero) of edges. The graph $E$ is called  \emph{finite} if $E^0$ and $E^1$ are finite sets.   

A  {\it path}  $p$ in $E$ is
a sequence $p=e_{1}e_{2}\cdots e_{n}$ of edges in $E$ such that
$r(e_{i})=s(e_{i+1})$ for $1\leq i\leq n-1$. We define $s(p) = s(e_{1})$, and $r(p) =r(e_{n})$. 
By definition, the \emph{length} $|p|$ of $p$ is $n$. We assign the length zero to vertices.  A \emph{closed path} (based at $v$) is a  path $p$ such that $s(p)=r(p)=v$. A \emph{cycle} (based at $v$) is a closed path $p=e_1 e_2 \cdots e_n$ based at $v$ such that $s(e_i)\neq s(e_j)$ for any $i\neq j$.  An edge $f\in E^1$ is called an \emph{exit} of a cycle $e_1\cdots e_n$ if there is an $1\leq i \leq n$ such that $s(f)=s(e_i)$ and $f\neq e_i$.

For every strongly connected cyclic component $C$ of $E$, we denote by $C^0$ the set of all vertices in $E$ which lie on $C$.


\begin{deff}[{\cite[Defenition 2.1]{BT}}]\label{sanddefgraph}
	A finite directed graph $E$ is called a \emph{sandpile graph} if $E$ has a unique sink (denote it by $s$), and for every $v\in E^0$ there is a  path $p$ in $E$ such that $s(p)=v$ and $r(p) = s$.  
\end{deff}

For clarification, we illustrate Definition \ref{sanddefgraph} by presenting the following example.

\begin{example}
	(1)  The following graph
	$$
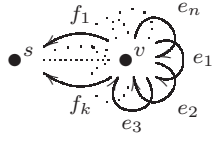
\xymatrix{
		\!\!\!    \bullet^s   &   \bullet^v \ar@{.}[l]  \ar@{.}@/_{5pt}/ [l]  \ar@/_{10pt}/ [l]_{f_1}  \ar@/^{10pt}/ [l]^{f_k}     \ar@{.}@(l,d) \ar@(ur,dr)^{e_{1}} \ar@(r,d)^{e_{2}} \ar@(dr,dl)^{e_{3}} 
		\ar@{.}@(l,u) \ar@(u,r)^{e_n}  \ar@{.}@(ul,ur)& 
	}$$	is a sandpile graph.
	
	(2) The following  graph
	$$\xymatrix{\bullet^{w}\ar[r]^{e}&
		\bullet^{v} \ar@(ul,ur)^f\ar[r]^g \ar@(dl,dr)_l& \bullet^{s}&\bullet^{u}\ar[l]_{h}}$$ is a sandpile graph.
\end{example}

\begin{deff}[{\cite[Defenition 2.8]{BT}}]\label{sandmonoid}
	Let $E$ be a sandpile graph. Then the {\it sandpile monoid} of $E$, denoted by $\text{SP}(E)$, is the free commutative monoid on a set of generators $E^0$, modulo relations given by  
	\begin{center}
		$s = 0$\quad and\quad $|s^{-1}(v)|v = \sum_{e\in s^{-1}(v)}r(e),$
	\end{center}
	for all $v\in E_{\reg}^0$. 
	
	Phrased another way, let $\langle v\mid v\in E^0\rangle$ be the free commutative monoid on a set of generators $E^0$ and $\langle s =0; |s^{-1}(v)|v = \sum_{e\in s^{-1}(v)}r(e) \text{ for all } v\in E_{\reg}^0 \rangle$ its congruence generated by the above relations. Then
	
	\[\text{SP}(E) = \frac{\langle v \mid v \in E^0\rangle}{\langle s = 0; |s^{-1}(v)|v= \sum_{e\in s^{-1}(v)}r(e)\text{ for all } v\in E^0\setminus \{s\}\rangle }.\]
\end{deff}

For clarification, we illustrate the sandpile monoid of a sandpile graph by presenting the following example. 

\begin{example}
	(1) Let $E_k$ be the sandpile graph which consists of two vertices and $k$ edges from the the vertex to the sinks, that means, $E_k$ is the graph pictured here.
	$$
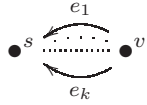
\xymatrix{
		\!\!\!    \bullet^s   &   \bullet^v \ar@{.}[l]  \ar@{.}@/_{5pt}/ [l]  \ar@/_{10pt}/ [l]_{e_1}  \ar@/^{10pt}/ [l]^{e_k}     & 
	}$$
	
	Then we have  $$\text{SP}(E) = \frac{\langle v, s\rangle}{(s=0; kv = s)}\cong \frac{\langle v\rangle}{(kv = 0)}\cong \mathbb{Z}_k.$$
	
	(2) Let $E$ be the sandpile graph:
	
	$$
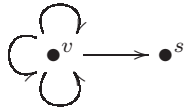
\xymatrix{\bullet^{v} \ar@(ul,ur) \ar@(lu,ld) \ar@(dl,dr) \ar[r]& \bullet^{s}}$$ 	\medskip
	
	Then we have  $$\text{SP}(E) = \frac{\langle v, s\rangle}{(s=0; 4v = 3v +s)}\cong \frac{\langle v\rangle}{(4v = 3v)},$$ that means, 
	\begin{center}
		$\text{SP}(E) = \{0, v, 2v, 3v\}$, with relation $4v = 3v$.
	\end{center}
	
	(3) Let $E$ be the sandpile graph
	$$
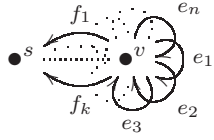
\xymatrix{
		\!\!\!    \bullet^s   &   \bullet^v \ar@{.}[l]  \ar@{.}@/_{5pt}/ [l]  \ar@/_{10pt}/ [l]_{f_1}  \ar@/^{10pt}/ [l]^{f_k}     \ar@{.}@(l,d) \ar@(ur,dr)^{e_{1}} \ar@(r,d)^{e_{2}} \ar@(dr,dl)^{e_{3}} 
		\ar@{.}@(l,u) \ar@(u,r)^{e_n}  \ar@{.}@(ul,ur)& 
	}$$
	
	Then we have  $$\text{SP}(E) = \frac{\langle v, s\rangle}{\langle s=0; (n+k)v = nv + ks \rangle}\cong \frac{\langle v\rangle}{\langle (n+k)v = nv \rangle},$$
	that means, 
	\begin{center}
		$\text{SP}(E) = \{0, v, \hdots,nv,\hdots, (n+k-1)v\}$, with relation $(n+k)v = nv$.	
	\end{center}
	\end{example}

It is worth mentioning that every  finite abelian group may be realized as a sandpile monoid (see \cite[Page 256]{CGGMS}), and  Chapman et.al. \cite[Theorems 5.1 and 5.4]{CGGMS} showed two infinite families of finite commutative monoids that cannot be realized as sandpile monoids of sandpile graphs.\medskip

Let $E$ be a sandpile graph. There is an explicit description of the congruence on the free commutative monoid $\langle v\mid v\in E^0\rangle$ given by the defining relations of $\text{SP}(E)$ in Definition \ref{sandmonoid}, as follows. A nonzero element of $\langle v\mid v\in E^0\rangle$ can be written uniquely up to permutation as $\sum_{i=1}^{n}k_iv_{i}$, where $v_{i}$ are distinct vertices and $k_i\in \mathbb N^+$. Define a binary relation
$\rightarrow_{1}$ on  the free commutative monoid $\langle v\mid v\in E^0\rangle$ by 

\begin{equation}\label{hfgtrgt655}
\sum_{i=1}^{n}k_iv_{i} \ \longrightarrow_{1}  \left\{
\begin{array}{lcl}
\sum_{i\neq
	j}k_iv_{i} &  & \text{if } v_j = s\\
\Big( \sum_{i\neq
	j}k_iv_{i} \Big) +(k_j-|s^{-1}(v_j)|)v_j+ \sum_{e\in s^{-1}(v_j)}r(e))&  & \text{if } v_j \neq s,
\end{array}%
\right.
\end{equation}
whenever $j\in \{1, \hdots, n\}$ and 
$k_j\geq |s^{-1}(v_j)|$. Let $\rightarrow$ be the transitive and reflexive closure of $\rightarrow_{1}$
on $\langle v\mid v\in E^0\rangle$.  Namely 
\begin{equation}\label{hfgtrgt6551}
a\rightarrow b    \ \  \ \text{ if }  a=b, \ \mbox{or}  \ a=a_0 \rightarrow_1 a_1 \rightarrow_1 \dots \rightarrow_1 a_k=b.
\end{equation}

\noindent
Finally, let   $\sim$   be the congruence on $\langle v\mid v\in E^0\rangle$ generated by the relation $\rightarrow$.   That means, $a\sim b$ in case  there is a string $a=a_0, a_1,\dots, a_n=b$ in $\langle v\mid v\in E^0\rangle$ such that $a_i\rightarrow_1 a_{i+1}$ or $a_{i+1}\rightarrow_1 a_{i}$ for each $0\leq i \leq n-1$. We then have 
$$\text{SP}(E)=\mathbb \langle v\mid v\in E^0\rangle/\sim.$$
\noindent
To avoid cumbersome equivalence class notation, it is standard (but not technically correct) to denote the elements of $\text{SP}(E)$ and the elements of $\langle v\mid v\in E^0\rangle$ using the same symbols.  For instance, we  will sometimes write $a=b$ in $\text{SP}(E)$ for  elements $a,b \in \mathbb \langle v\mid v\in E^0\rangle$. 

The {\it support} of an element $\alpha$ in $\mathbb \langle v\mid v\in E^0\rangle$, denoted by $\text{supp}(\alpha)\subseteq E^0$, is the set of basis elements appearing in the canonical expression of $\alpha$.\medskip

The following proposition shows that sandpile monoids have ``Confluence Property", which was established by Abrams and the first author in \cite[Lemma 3.1]{GeneRooz}. 

\begin{prop}[{\cite[Lemma 3.1]{GeneRooz}}]\label{Conflupro}
Let $E$ be a sandpile graph and $a, b\in\rm{SP}$$(E)$ nonzero elements. Then $a = b$ in $\rm{SP}$$(E)$  if and only if there is an element $c\in \mathbb \langle v\mid v\in E^0\rangle$ such that $a \rightarrow c$ and $b\rightarrow c$. 
\end{prop}

Let $E$ be a graph and $H\subseteq E^0$. We say $H$ is  \emph{hereditary} if for any $e \in E^1$,  $s(e)\in H$ implies $r(e)\in H$. We say $H$ is {\it saturated} if  whenever $v\in E_{\reg}^0$ with the property that $r(s^{-1}(v))\subseteq H$, then $v\in H$. The {\it hereditary saturated closure} of $H$, denoted by $\overline{H}$, is the smallest hereditary and saturated subset of $E^0$ containing $H$.	We denote by $\mathcal{H}_E$ the set of all nonempty saturated hereditary subsets of $E^0$. If $E$ is a sandpile graph, then the set of all vertices in $E$ which do not connect to any cycle in $E$ is the smallest element in $(\mathcal{H}_E, \subseteq)$, and so $(\mathcal{H}_E, \subseteq)$ is a complete lattice with $H \vee H' = \overline{H \cup H'}$ and $H \wedge H' = H\cap H'$.

For any $H\in \mathcal{H}_E$, we denote by $E_H$ the {\it restriction graph}
	\begin{center}
		$E_H^0 := H$, \quad\quad $E_H^1 := \{e\in E^1\mid s(e)\in H\}$
	\end{center}	
	and the source and range functions in $E_H$ are exactly the source and range functions in $E$, restricted to $H$.

The following remark is easy and useful.

\begin{remark}\label{sathererem}
	Let $E$ be a sandpile graph. Then $E_H$ is a sandpile graph for all $H\in \mathcal{H}_E$.	
\end{remark}
\begin{proof}
	Let $H$ be a nonempty saturated hereditary subset of $E^0$. Then, there exists a vertex $v\in E^0$ such that $v\in H$. Since $E$ is a sandpile graph, there exists a path $p$ in $E$ such that $v= s(p)$ and $s = r(p)$. This shows that $s\in H$ (since $H$ is hereditary). It is obvious that for every vertex $w$ in $H$, there is a path $q$ in $E_H$ such that $w= s(q)$ and $r(q)= s$. Thus, $E_H$ is a sandpile graph.	
\end{proof}	

As an application of Proposition \ref{Conflupro}, we obtain the following useful proposition.

\begin{prop}\label{sandsubmon}
Let $E$ be a sandpile graph and $H\in \mathcal{H}_E$. Then, the map
	$\psi_H: \rm{SP}$$(E_H)\longrightarrow \rm{SP}$$(E)$, defined by $x \longmapsto x$, is an injective homomorphism of monoids. Consequently,
	$\rm{SP}$$(E_H)$ is a submonoid of $\rm{SP}$$(E)$.	
\end{prop}
\begin{proof} We note that $E_H$ is a sandpile graph, by Remark \ref{sathererem}.
	Let $a, b\in \text{SP}(E_H)$ such that $a = b$ in $\text{SP}(E)$. By Proposition \ref{Conflupro}, there is an element $c\in \langle v\mid v\in E^0\rangle$ such that $a \rightarrow c$ and $b\rightarrow c$. Since $H$ is hereditary, all the transformations occur in $E_H$, and so $c\in \langle v\mid v\in H^0\rangle$. This implies that $a = b$ in $\text{SP}(E_H)$. Hence, $\psi_H$ is  an injective homomorphism of monoids, thus finishing the proof.
\end{proof}	

Let $(M, +)$ be a commutative semigroup. Define an an equivalence relation $\sim$ on $M\times M$ by setting: $(m_1, m_2)\sim (n_1, n_2)$ if $m_1 + n_2 +x = m_2 + n_1 +x$ for some $x\in M$. Let $G(M)$ denote the equivalence classes in $M\times M$ under $\sim$ (we denote an individual class by $[\ ]_0$), and define $+$ on $G(M)$ as follows: \[[(m_1, m_2)]_0 + [(n_1, n_2)]_0 = [(m_1 +n_1, m_2+n_2)]_0.\]
It is not hard to check that $(G(M), +)$ is abelian group, and called the {\it Grothendieck group} of $M$. In particular, the identity of $G(M)$ is $(m, m)]_0$ for any $m\in M$, and $-[(m, n)]_0 = (n, m)]_0$ for all $m, n\in M$. The map $\phi: M\longrightarrow G(M)$, $m\longmapsto [m, m +x]_0$ for any $x\in M$, is a monoid homomorphism. Moreover, $G(M)$ satisfies the Universal Property: for any abelian group $G'$ and for any monoid homomorphism $\psi: M\longrightarrow G'$, there is a unique group homomorphism $\lambda: G(M)\longrightarrow G'$ such that $\lambda \phi = \psi$. If, in addition, $M$ is finite, then $G(M)$ is  isomorphic to the smallest ideal of $M$. In particular, $G(M) \cong \bigcap_{x\in M}(x + M)$ (see \cite[Lemma 2.5]{GeneRooz}).

\begin{deff}[{\cite{BT, CGGMS}}]\label{sandpilegrp}
	For any sandpile graph $E$, the {\it sandpile group} $G(E)$ of $E$ is the Grothendieck group of $\rm{SP}$$(E)$.
\end{deff}

Let $E$ be a sandpile graph and $a, b\in \text{SP}(E)$. We say that $a$ is {\it accessible} from $b$ if $a = b +x$ for some $x\in \text{SP}(E)$. We say that $a$ is {\it recurrent} if $a$ is accessible from every element $c\in \text{SP}(E)$ (see, e.g., \cite[Page 7]{BT}).

\begin{prop}[{\cite[Corollary 3.2]{BT}}]\label{recur}
	For every sandpile graph $E$, $G(E)$ is exactly the set of all recurrent elements of $\rm{SP}$$(E)$. 	
\end{prop}
\begin{proof}
Since $\text{SP}(E)$ is a finite commutative monoid and by \cite[Lemma 2.5]{GeneRooz}, $G(E)$ is the smallest ideal of $\text{SP}(E)$, and so $c + G(E) = G(E)$ for all $c\in \text{SP}(E)$,
	which yields the statement.
\end{proof}

Recall (see, e.g., \cite[Page 131]{TheBook}) that an {\it order-ideal} of a commutative monoid $M$ is a submonoid $I$ of $M$ such that, for any $x, y\in M$, if $x+y \in I$ then $x, y\in I$. An order-ideal may also be described as a submonoid $I$ of $M$ which is hereditary with respect to the canonical preorder $\le$ on $M$: $x\le y$ and $y\in I$ imply $x\in I$, where  the preorder $\le$ on $M$ is defined by
setting $x\le y$ if  $y = x + m$ for some $m\in M$. For each $X\subseteq M$, the order-ideal of $M$ generated by $X$ is the set \[\langle X \rangle :=\{m\in M\mid m\le \sum^n_{i=1}x_i,\ n\in \mathbb{N}^+,\ x_i\in X\}.\]
The set $\mathcal{L}(M)$ of all order-ideals of $M$ forms a complete lattice such that the join of two elements $I, J$ is exactly the order-ideal of $M$ generated by $I +J$. 

An element $m\in M$ is an {\it idempotent} if $m + m = m$. We denote by $\text{Idem}(M)$ the set of all idempotents of $M$. Then $(\text{Idem}(M), \le)$ is a join-semilattice with the join of two elements $m, m'$ is $m+m'$. If, in addition, $M$ is finite, then $(\text{Idem}(M), \le)$ is a lattice with $m \wedge m' = \sum_{x \le m, m'}x$.

We are now in a position to establish the main result of this section describing the sandpile monoid $\text{SP}(E)$ and the sandpile group $G(E)$ of a sandpile graph $E$ in terms of saturated hereditary subsets of $E$.

\begin{thm}\label{sandmonstruc}
Let $E$ be a sandpile graph and $S_E$ the set of all vertices in $E$ which do not connect to any cycle in $E$. Then the following statements hold:
	
$(1)$ $\rm{SP}$$(E_H)$ is a submonoid of $\rm{SP}$$(E_{H'})$ for all $H$ and $H'\in \mathcal{H}_E$ with $H\subseteq H'$;
	
$(2)$ $\rm{SP}$$(E_{S_H}) = Z(\rm{SP}$$(E))$, where $Z(\rm{SP}$$(E))$ is the unit group of $\rm{SP}$$(E)$;
	
$(3)$ For all $H\in \mathcal{H}_E$, $\rm{SP}$$(E_H)$ is an order-ideal of $\rm{SP}$$(E)$  containing $Z(\rm{SP}$$(E))$, and $\rm{SP}$$(E_H) = \langle H\rangle$;
	
$(4)$ For every order-ideal $I$ of $\rm{SP}$$(E)$, $I = \rm{SP}$$(E_H)$, where $H= I\cap E^0\in \mathcal{H}_E$;
	
$(5)$ $\mathcal{H}_E\cong \mathcal{L}(\rm{SP}$$(E))$ as lattices;
	
$(6)$ $G(E) = \varinjlim_{H\in \mathcal{H}_E} G(E_H)$, where $G(E_H) := G(\rm{SP}$$(E_H))$;
	
$(7)$ For all $H\in \mathcal{H}_E$, $G(E_H)= z_H + \rm{SP}$$(E_H)$, where $z_H :=\sum_{x\in \rm{Idem}(\rm{SP}(E_H))}x$. Consequently, $G(E_H)$ is a maximal subgroup of $\rm{SP}$$(E)$.
\end{thm}
\begin{proof}
(1) Let $H$ and $H'$ be two elements of $\mathcal{H}_E$ with $H\subseteq H'$. Then, $H$ is a saturated hereditary subset of $(E_{H'})^0$. By Proposition \ref{sandsubmon}, there exists an injective monoid homomorphism $\psi_{H, H'}: \text{SP}(E_H)\longrightarrow \text{SP}(E_{H'})$ such that $\psi_{H, H'}(a) = a$ for all $a\in \text{SP}(E_H)$, and so $\rm{SP}$$(E_H)$ may be considered as a submonoid of $\rm{SP}$$(E_{H'})$.

(2) It follows from \cite[Proposition 3.3]{GeneRooz}.
	
(3) Let $H$ be an element in $\mathcal{H}_E$. We then have $S_E\subseteq H$. By items (1) and (2), $\text{SP}(E_H)$ is a submonoid of $\text{SP}(E)$ containing $Z(\text{SP}(E))$. Assume that $x$ any $y$ are elements in $\text{SP}(E)$ with $x\le y$ and $y\in \text{SP}(E_H)$. Then, $y = x +z$ in $\text{SP}(E)$ for some $z\in \text{SP}(E)$. If $y\in Z(\text{SP}(E))$, then $x\in Z(\text{SP}(E))$ and so $x\in \text{SP}(E_H)$, as desired.
	
Considering the case when $y\notin Z(\text{SP}(E))$.
By Proposition \ref{Conflupro}, there exists an element $c\in \mathbb \langle v\mid v\in E^0\rangle$ such that $y \rightarrow c$ and $x+z\rightarrow c$. Since $y\in \text{SP}(E_H)$, $\text{supp}(y)\subseteq E^0_H$, and so $\text{supp}(c)\subseteq E^0_H$. Assume that $x\notin \text{SP}(E)$. Then, there exists a vertex $v\in \text{supp}(x)$ such that $v\notin E^0_H$; in particular, $v\notin S_E$. Since $H$ is saturated, $v$ must connect to a cycle which is not in $E_H$. Hence, any possible transformations of $v$ or its multiple would give a vertex on a cycle which is not in $E_H$ and subsequently any further transformations always contain a vertex on a cycle which is not in $E_H$. This shows that $x + y$ cannot be transformed to $c$, a contradiction, and so $x\in \text{SP}(E)$. Thus $\text{SP}(E_H)$ is an order-ideal of $\text{SP}(E)$ and $\rm{SP}$$(E_H) = \langle H\rangle$.
	
(4) Let $I$ be an order-ideal of $\text{SP}(E)$. Since $0\in I$ and $I$ is an order-ideal, $Z(\text{SP}(E))\subseteq I$. Let $H := I \cap E^0$. By item (2), $S_E\subseteq H$. We claim that $H\in \mathcal{H}_E$. Indeed, let $e\in E^1$ with $v:=s(e)\in H$. We then have $v\in I$ and $\sum_{f\in s^{-1}(v)}r(f)=|s^{-1}(v)|v \in I$, and so $r(f)\in I$ for all $f\in s^{-1}(v)$ (since $I$ is an order-ideal of $\text{SP}(E)$); in particular, $r(e)\in I$. Hence, $H$ is hereditary.
	
Let $v$ be a regular vertex in $E$ with $r(s^{-1}(v))\subseteq I$. Then, $|s^{-1}(v)|v=\sum_{f\in s^{-1}(v)}r(f) \in I$, and so $v\in I$. Therefore, $H$ is saturated, showing the claim.
	
We next prove that $I= \text{SP}(E_H)$. Since $H\subseteq I$ and $I$ is a submonoid of $\text{SP}(E)$, we have $\text{SP}(E_H)\subseteq I$. Let $x$ be a nonzero element in $I$. Since $I$ is an order-ideal of $\text{SP}(E)$, we must have $\text{supp}(x)\subseteq I$, and so $\text{supp}(x)\subseteq H$. This implies that $x\in \text{SP}(E_H)$, that means, $I\subseteq \text{SP}(E_H)$, and so $I= \text{SP}(E_H)$.

(5) Let $\phi: \mathcal{H}_E\longrightarrow \mathcal{L}(\text{SP}(E))$ be the map defined by: $H\longmapsto\phi(H)= \text{SP}(E_H)$, and let $\psi: \mathcal{L}(\text{SP}(E))\longrightarrow\mathcal{H}_E$ be the map defined by: $I\longmapsto \psi(I) = I\cap E^0$. Then, it is obvious that these maps are order-preserving. Also, we have $\psi(\phi(H)) = \psi(\text{SP}(E_H)) = H$ for all $H\in \mathcal{H}_E$ (by item (3)), that means, $\psi\phi = id_{\mathcal{H}_E}$. On the other hand, by item (4), we obtain that $\phi\psi = id_{\mathcal{L}(\text{SP}(E))}$. Therefore, $\phi$ and $\psi$ are order-preserving mutually inverse maps.
	
(6) It is clear that $\mathcal{H}_E$ is directed, since  the join of two elements $H$ and $H'$ is exactly the hereditary saturated closure of $H \cup H'$. 

Let $H$ and $H'$ be two elements of $\mathcal{H}_E$ with $H\subseteq H'$. By item (1), there exists a group homomorphism $\phi_{H, H'}: \text{G}(E_H)\longrightarrow \text{G}(E_{H'})$ such that $\phi_{H, H'}([(a, b)]_0) = [(a, b)]_0$ for all $a, b\in \text{SP}(E_H)$. It obvious that $$\phi_{H, H''} = \phi_{H', H''}\circ\phi_{H, H'}$$ for all $H, H', H''\in \mathcal{H}_E$ with $H\subseteq H'\subseteq H"$.
This shows that $$\varinjlim_{H\in \mathcal{H}_E} \text{G}(E_H) = \bigsqcup_{H\in \mathcal{H}_E}\text{G}(E_H)/\equiv,$$ 	where $\equiv$ is the equivalence relation on $\bigsqcup_{H\in \mathcal{H}_E}\text{G}(E_H)$ defined as follows:  for $[(a, b)]_0\in \text{G}(E_H)$ and $[(a', b')]_0\in \text{G}(E_{H'})$, $[(a, b)]_0 \equiv [(a', b')]_0$ if and only if there exists an element $H''\in \mathcal{H}_E$ such that $H, H' \subseteq H''$ and $\phi_{H, H''}([(a, b)]_0) = \phi_{H', H''}([(a', b')]_0)$ in $\text{G}(E_{H''})$. Then, it is straightforward to see that $G(E) = \varinjlim_{H\in \mathcal{H}_E} G(E_H)$.
	
(7) By Proposition \ref{recur}, we have $G(E_H)\subseteq z_H + \text{SP}(E_H)$. Conversely, let $a$ and $b$ be two elements of $\text{SP}(E_H)$. Since $\text{SP}(E_H)$ is a finite commutative monoid and by \cite[Theorem 2.12]{CGGMS}, $m b$ is an idempotent in $\text{SP}(E_H)$ for some $m\in \mathbb{N}^+$. We then have $mb +x = z_H$ for some $x\in \text{SP}(E_H)$, and so $$z_H + a = mb + x + a = b + (m-1)b +x +a.$$ This shows that $z_H +a$ is recurrent. By Proposition \ref{recur}, $z_H + \text{SP}(E_H)\subseteq G(E_H)$, and so $z_H + \text{SP}(E_H)= G(E_H)$.
	
We claim that $G(E_H) = Z(z_H + \text{SP}(E))$. Indeed, since $G(E_H)$ is a group with the identity $z_H$, $G(E_H) \subseteq Z(z_H + \text{SP}(E))$. Conversely, let $a\in Z(z_H + \text{SP}(E))$. We then have $a + b = z_H$ for some $b\in \text{SP}(E)$. Since $z_H\in \text{SP}(E_H)$ and $\text{SP}(E_H)$ is an order-ideal of $\text{SP}(E)$ (by Item (3)), we have $a\in \text{SP}(E_H)$, and so $a = z_H + a \in z_H + \text{SP}(E_H)= G(E_H)$. This implies that $Z(z_H + \text{SP}(E))\subseteq G(E_H)$, thus finishing the proof.
\end{proof}

As a corollary of Theorem \ref{sandmonstruc}, we obtain the following interesting result.

\begin{cor}\label{sat-her-lat-iso}
Let $E$ and $F$ be two sandpile graphs. If $\rm{SP}$$(E)\cong \rm{SP}$$(F)$ as monoids, then $\mathcal{H}_E\cong \mathcal{H}_F$ as lattices.	
\end{cor}
\begin{proof}
	Assume that $\text{SP}(E)\cong \text{SP}(F)$ as monoids. We then obtain that $\mathcal{L}(\text{SP}(E))\cong \mathcal{L}(\text{SP}(F))$ as lattices.	By Theorem \ref{sandmonstruc} (5), we immediately obtain that $\mathcal{H}_E\cong \mathcal{H}_F$ as lattices, thus finishing the proof.
\end{proof}	

\section{Idempotents and archimedean classes of sandpile monoids} \label{weightedgraphssec3}
In this section, based essentially on Babai and Toumpakari's result \cite[Theorem 4.13]{BT} and Theorem \ref{sandmonstruc}, we describe all idempotents and all order-ideals of the sandpile monoid $\text{SP}(E)$ of a sandpile graph $E$ in terms of the combinatorics of $E$. In particular, we obtain that they are not only isomorphic to each other as lattices, but also isomorphic to the lattice of all nonempty saturated hereditary subsets of $E$ (Theorem \ref{idem-fiter-here-ordideal}). Also,  we prove that a sandpile group $G(E)$ is the direct limit of the Grothendieck groups of  archimedean classes of $\text{SP}(E)$ and obtain that the Grothendieck group of an archimedean class $[x]$ is exactly the sandpile group of the sandpile monoid arising from the support of $x$, as well as show that all maximal subgroups of $\text{SP}(E)$ are exactly the Grothendieck groups of  archimedean classes of $\text{SP}(E)$ (Theorem \ref{arch-theo}).\medskip

Let $(P,\leq)$ be a partially ordered set. We say a nonempty subset $I$ of $P$  is an \emph{ideal} of $P$ if $x\in I$ and $y\leq x$ then $y \in I$. We say a set $F$ is a \emph{filter} of $P$ if $x\in F$ and $x\leq y$ implies $y\in F$. 

Let $E$ be an arbitrary graph. Let $\mathcal C_E$ be the set of all strongly connected cyclic components of $E$, i.e.,  strongly connected components which contain cycles. Consider $\mathcal C_E$ as a partially ordered set with $C\leq C'$ if there is a path connecting $C$ to $C'$. We denote by $\mathcal{F}_E$ the set of all filters of $\mathcal{C}_E$. It is obvious that $(\mathcal{F}_E, \subseteq)$ is a complete lattice such that $F\vee F' = F \cup F'$ and $F\wedge F' = F\cap F'$ for all $F, F'\in \mathcal{F}_E$.

\begin{lemma}\label{fil-hered}
Let $E$ be a sandpile graph, $S_E$ the set of all vertices in $E$ which do not connect to any cycle in $E$, $\mathcal{F}_E$ the complete lattice of all filters of $\mathcal{C}_E$. Then, the corresponding $\phi_E:  \mathcal{F}_E\longrightarrow \mathcal{H}_E$, defined by $\emptyset \longmapsto S_E$ and $\emptyset\neq F \longmapsto \overline{\cup_{C\in F}C^0}$, is a lattice isomorphism, where $C^0$ is the set of all vertices of $C$.
\end{lemma}
\begin{proof}
It is obvious that $\phi_E$ is a lattice homomorphism. Let $H$ be a nonempty saturated hereditary subsets of $E^0$. We denote by $F$ the set of 	all strongly connected cyclic components of the restriction graph $E_H$ of $E$ onto $H$. Assume that $C_1, C_2\in \mathcal{C}_E$ with $C_1\le C_2$ and $C_1\in F$. Then, we have that $C_1$ is a strongly connected cyclic components of $E_H$. Since $C_1$ connects to $C_2$ and $H$ is hereditary, we obtain that $C_2$ is also a strongly connected cyclic components of $E_H$, that means, $C_2\in F$. Hence, $F$ is a filter of $\mathcal{C}_E$, and $\phi_E(F) = H$. This implies that $\phi_E$ is surjective.
	
Let $F_1$ and $F_2$ be elements of $\mathcal{F}_E$ with $\phi_E(F_1) = \phi_E(F_2)$. We then have $H_1 := \overline{\cup_{C\in F_1}C^0}=\overline{\cup_{C\in F_2}C^0} =: H_2$, and so the graphs $E_{H_1}$ and $E_{H_2}$ are the same. Since $H_i$ is saturated hereditary, the set of all strongly connected cyclic components of $E_{H_i}$ is exactly $F_i$. Therefore, $F_1 = F_2$, and so $\phi_E$ is injective, thus finishing the proof.
\end{proof}	

For clarification, we illustrate Lemma \ref{fil-hered} by presenting the following example.

\begin{example}
	Let $E$ be the following pictured graph
	$$
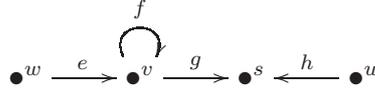
\xymatrix{\bullet^{w}\ar[r]^{e}&
		\bullet^{v} \ar@(ul,ur)^f\ar[r]^g& \bullet^{s}&\bullet^{u}\ar[l]_{h}}$$
	Then we have $S_E = \{s, u\}$ and $\mathcal{H}_E =\{S_E, E^0\}.$ Also, we have $\mathcal{F}_E = \{\emptyset, \{f\}\}$. Therefore, the lattice isomorphism introduced in Lemma \ref{fil-hered} is defined by: $\phi_E(\emptyset) = S_E$	and  $\phi_E(\{f\}) = E^0$.
\end{example}

Let $E$ be a sandpile graph and $S_E$ the set of all vertices in $E$ which do not connect to any cycle in $E$. Let $\mathcal C_E$ be the set of all strongly connected cyclic components of $E$. Let $\mathcal H^s_E= \{S_E, H_C\mid C\in \mathcal{C}_E\}$, where $H_C=\overline{C^0}$  is the hereditary and saturated closure of $E^0$ containing $C^0$.
We note that $(\mathcal H^s_E, \subseteq)$ is a partially ordered set with the smallest element $S_E$. We denote by $\text{Ideal}(\mathcal H^s_E)$ the set of all ideals of $\mathcal H^s_E$.

\begin{lemma}\label{fil-ideal}
Let $E$ be sandpile graph and $S_E$ the set of all vertices in $E$ which do not connect to any cycle in $E$. Then, the corresponding $\psi_E:  \mathcal{F}_E\longrightarrow \rm{Ideal}(\mathcal H^s_E)$, defined by $\emptyset \longmapsto \{S_E\}$ and $\emptyset\neq F \longmapsto H_F := \{S_E, H_C\mid C\in F\}$, is a lattice isomorphism.	
\end{lemma}
\begin{proof} We first claim that  $C \leq  C'$ if and only if $H_C \subseteq H_{C'}$ for all $C, C'\in \mathcal{C}_E$. Indeed, if $C \leq C'$, then there is a path connecting $C$ to $C'$. Thus the hereditary and saturated subset generated by $C$ contains all the vertices in $C'$. Therefore $H_C \subseteq H_{C'}$. 
	
Suppose $H_C \subseteq H_{C'}$ but there is no path from $C$ to $C'$. Denote by $X_0=T(C)$ the hereditary closure of $C$. Clearly $C' \cap X_0=\varnothing$.  Recall the construction of hereditary and saturated set $H_C=\bigcup_{n\geq 0} X_n$ from ~\cite[Lemma 2.0.7]{TheBook}, where $$X_{i+1} = S(X_i) := \{v\in E^0_{\reg}\mid \{r(e) \mid e \in s^{-1}(v)\}\subseteq X_i\} \cup X_i.$$
Let $u\in (C')^0$. Since  $\overline{(C')^0}=H_{C'} \subseteq H_C$, there is $k\geq 0 $ such that $u \not \in X_k$ but $u \in X_{k+1}$. That means, all the edges emitting from $u$ land in $X_k$. But $X_k$ is hereditary and $u$ is on a cycle (since $C'$ is a strongly connected cyclic component). This gives that $u\in X_k$, a contradiction. Thus there is a path connecting $C$ to $C'$, that means, $C\le C'$, showing the claim.	The claim shows that $\psi_E$ is a lattice isomorphism, thus finishing the proof.
\end{proof}	

For clarification, we illustrate Lemma \ref{fil-ideal} by presenting the following example.

\begin{example}
Let $E$ be the following pictured graph
	$$
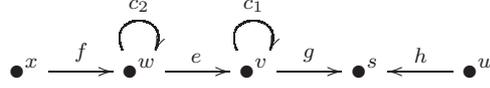
\xymatrix{\bullet^{x}\ar[r]^{f}&\bullet^{w}\ar@(ul,ur)^{c_2}\ar[r]^{e}&
		\bullet^{v} \ar@(ul,ur)^{c_1}\ar[r]^g& \bullet^{s}&\bullet^{u}\ar[l]_{h}}$$
	Then we have $S_E = \{s, u\}$ and $\mathcal{F}_E = \{\emptyset, \{c_1\}, \{c_1, c_2\}\}$. Also, we have  $H_{c_1}=\overline{c^0_1} = \overline{\{v\}} = \{v, s, u\}$ and $H_{c_2} =\overline{c^0_2} = \overline{\{w\}} = E^0$, and so $\mathcal{H}^s_E =\{S_E, H_{c_1}, H_{c_2}\}.$ Therefore, the lattice isomorphism introduced in Lemma \ref{fil-ideal} is defined by: $\psi_E(\emptyset) = \{S_E\}$,  $\psi_E(\{c_1\}) = \{S_E, H_{c_1}\}$ and $\psi_E(\{c_1, c_2\}) = \{S_E, H_{c_1}, H_{c_2}\}$.
\end{example}

In \cite{BT}, Babai and Toumpakari provided us with a method to read the idempotent structure of the sandpile monoid $\text{SP}(E)$ of a sandpile graph $E$ via the cycle structure of $E$. 

\begin{thm}[{\cite[Theorem 4.13]{BT}}]\label{BT-thm}
For every sandpile graph $E$, the correspondence $\delta_E: \rm{Idem}(SP(E))\longrightarrow \mathcal{F}_E$, defined by $\delta_E(0) = \emptyset$ and $\delta_E(x)$ is the set of all  strongly connected cyclic components of $E_H$, where $H = \overline{\rm{supp}(x)}$, for all $0\neq x\in \rm{Idem}(SP(E))$, is a lattice isomorphism.
\end{thm}

We are now in a position to establish the first main result of this section describing all idempotents and all order-ideals of $\text{SP}(E)$ in terms of the combinatorics of $E$, which yields that they are isomorphic to each other as lattices.

\begin{thm}\label{idem-fiter-here-ordideal}
For every sandpile graph $E$, there is a lattice isomorphism between the following lattices: 
	
$(1)$ The lattice $\rm{Idem}($$\rm{SP}$$(E))$ of all idempotents of $\rm{SP}$$(E)$;
	
$(2)$ The lattice $\mathcal{F}_E$ of all filters of $\mathcal{C}_E$;
	
$(3)$ The lattice $\mathcal{H}_E$ of all nonempty saturated hereditary subsets of $E^0$;
	
$(4)$ The lattice $\rm{Ideal}(\mathcal H^s_E)$ of all ideals of $\mathcal{H}^s_E$;
	
$(5)$ The lattice $\mathcal{L}(\rm{SP}$$(E))$ of all order-ideals of $\rm{SP}$$(E)$.
	
\end{thm}
\begin{proof}
By Theorem \ref{BT-thm}, $\rm{Idem}($$\rm{SP}$$(E))\cong \mathcal{F}_E$ as lattices via the isomorphism $\delta_E$. By Lemma \ref{fil-hered}, $\mathcal{F}_E\cong \mathcal{H}_E$ as lattices via the isomorphism $\phi_E$. By Lemma \ref{fil-ideal}, $\mathcal{F}_E\cong \rm{Ideal}(\mathcal H^s_E)$ as lattices via the isomorphism $\psi_E$. By Theorem \ref{sandmonstruc} (5), we obtain that $\mathcal{L}(\text{SP}(E))\cong \mathcal{H}_E$ as lattices. From these observations, we immediately obtain the theorem, thus finishing the proof.
\end{proof}

We may construct an explicit isomorphism between the lattice of all idempotents of $\text{SP}(E)$ and the lattice of all nonempty saturated hereditary subsets of $E$.

\begin{cor}\label{idem-sathere-rem}
Let $E$ be a sandpile graph and $S_E$ the set of all vertices in $E$ which do not connect to any cycle in $E$. Then, the map $\varsigma_E: \rm{Idem}(SP(E))\longrightarrow \mathcal{H}_E$, defined by $0\longmapsto S_E$ and $0\neq x \longmapsto \overline{\rm{supp}(x)}$, is a lattice isomorphism.	
\end{cor}
\begin{proof}
It immediately follows from the proof of Theorem \ref{idem-fiter-here-ordideal} that $\varsigma_E = \phi_E\delta_E$.
\end{proof}	



For any commutative monoid $(M, +)$, its canonical preoder induces the following equivalence relation on $M$, the now-so-called {\it archimedean equivalence relation}.

\begin{fact}\label{Archorder}
	For every commutative monoid $M$, the  binary relation, defined by 
	\begin{center}
		$x\asymp y$ if and only if $x\le my$ and $y\le nx$ for some $m, n\in \mathbb{N}$, 
	\end{center}
	is an equivalence relation on $M$.	
\end{fact}
\begin{proof}
It is obvious that $\asymp $ is reflexive and symmetric.
	
Let $x, y, z$ be elements of $M$ such that $x\asymp y$ and $y\asymp z$. Then $x\le m y$, $y\le n x$, $y\le k z$ and $z\le l y$ for some $m, n, k, l\in \mathbb{N}$, and so $x\le (mk)z$ and $z\le (ln)x$, that means, $x\asymp z$. This implies that  $\asymp$ is transitive, and so $\asymp$ is an equivalence relation on $M$. 	
\end{proof}

For any commutative monoid $M$ and for any $x\in M$, the $\asymp$-equivalence class of $x$, denoted by $[x]$, is called  the {\it archimedean class} of $x$; that means, $[x] = \{y\in M\mid y\asymp x\}$.\medskip

We are now in a position to establish the second main result of this section describing the sandpile group and all maximal subgroups of a sandpile monoid via archimedean classes.

\begin{theorem}\label{arch-theo}
Let $E$ be a sandpile graph and $S_E$ the set of all vertices in $E$ which do not connect to any cycle in $E$. Then, the following statements hold:
	
$(1)$ For all $x\in \rm{SP}$$(E)$, $[x]$ is a subsemigroup of $\rm{SP}$$(E)$ containing a unique idempotent;
	
$(2)$ $\rm{SP}$$(E) =\bigsqcup_{x\in \rm{Idem}(SP(E))}[x]$; 
	
$(3)$ $G(E) = \varinjlim_{x\in \rm{Idem}(SP(E))} G([x])$, where $G([x])$ is the Grothendieck group of $[x]$;

$(4)$ For all $x\in \rm{Idem}(SP(E))$, \[G[x] \cong x + [x]= \begin{cases}
G(E_{S_E}) & \text{if } x = 0,\\
G(E_H) & \text{if } x \neq 0, \end{cases}\]
where $H= \overline{\rm{supp}(x)}$. In particular,  $G(E) \cong e_{\max} + [e_{\max}],$ where $e_{\max} = \sum_{x\in \rm{Idem}(SP(E))}x$;
	
	
$(5)$ For every maximal subgroup $G$ of $\rm{SP}$$(E)$, $G \cong x + [x]$ for some $x\in \rm{Idem(SP(E))}$. Consequently, the set of all maximal subgroups of $\rm{SP}$$(E)$ is exactly the set $$\{x + [x]\mid x\in \rm{Idem(SP(E))}\}.$$
\end{theorem}
\begin{proof}
(1) Let $y$ and $z \in [x]$. We then have $x \le n y$, $y\le m x$, $z\le k x$ and $x\le lz$ for some $m, n, k, l\in \mathbb{N}$, and so $ y + z \le (m + k)x$ and $x \le ny \le n(y +z)$. This implies that $y +z \in [x]$, an hence $[x]$	is a subsemigroup of $\text{SP}(E)$. Since $[x]$ is finite and by \cite[Lemma 2.10]{CGGMS}, $t x$ is an idempotent for some $t\in \mathbb{N}^+$; that means, $[x]$ always contains an idempotent.
	
Assume that $e$ and $f$ are two idempotents of $[x]$. We then have $e\asymp f$, and so $e = f + x$ and $f = e +y$ for some $x, y\in \text{SP}(E)$. This implies that $e + f = f + x + f = f + x = e$ and $e + f = e + e + y = e + y = f$, and so $e = f$, proving (1).
	
(2) By Fact \ref{Archorder} and Item (1), we immediately obtain that  $\text{SP}(E) =\bigsqcup_{x\in \rm{Idem}(SP(E))}[x]$.
	
(3) Let $e$ and $f$ be two idempotents of $\text{SP}(E)$ with $e\le f$. We then have $f = e + x$ for some $x\in \text{SP}(E)$, and so $e + f = e + e + x = e+ x = f$.
	
Let $x \in [e]$. Then, $e = x + a$ and $n x= e + b$ for some $a, b\in \text{SP}(E)$ and $n\in \mathbb{N}^+$, and so $f = f + e = f + (x + a) = (f + x) + a$, that means, $f + x\le f$. It is obvious that $f\le f + x$, and so $f + x \in [f]$. Therefore, the corresponding $\phi_{e, f}: [e] \longrightarrow [f]$, defined by $x \longmapsto f + x$, is a map. Moreover, we have $$\phi_{e, f}(x + y) = f + x + y = f + x + f + y = \phi_{e, f}(x) + \phi_{e, f}(y)$$ for all $x, y\in [e]$, and so $\phi_{e, f}$ is a semigroup homomorphism. Also, it is not hard to check that
$\phi_{f, g}\circ\phi_{e, f} = \phi_{e, g}$ for all $e, f, g\in \rm{SP(E)}$ with $e\le f\le g$.
	
	Every semigroup homomorphisms $\phi_{e, f}$ induces a group homomorphism $\phi_{e, f}: G([e]) \longrightarrow G([f])$ such that $\phi_{e, f}([(x, y)]_0) = [(\phi_{e, f}(x), \phi_{e, f}(y))]_0$ for all $[(x, y)]_0\in G([e])$. We note that $\phi_{e, e}([(x, y)]_0) = [(\phi_{e, e}(x), \phi_{e, e}(y))]_0 = [(e+ x, e+y)]_0 =[(x, y)]_0 + [(e, e)]_0 = [(x, y)]_0$, and so $\phi_{e, f} = id_{G([e])}$. Therefore, the pair $(G[e], \phi_{e, f})$ is a direct system of finite abelian groups over $\rm{Idem}(SP(E))$. Let $\psi_e: G([e])\longrightarrow \varinjlim_{x\in \rm{Idem}(SP(E))} G([x])$ be the canonical homomorphism.
	
	For every $e\in \rm{Idem}(SP(E))$, since $[e]$ is a subsemigroup of $\text{SP}(E)$, there exists a group homomorphism $\alpha_e: G([e])\longrightarrow G(E)$ such that $\alpha_e([(x, y)]_0) = [(x, y)]_0$ for all $[(x, y)]_0\in G([e])$.  Let $e$ and $f$ be two idempotents of $\text{SP}(E)$ with $e\le f$. Then, for each $[(x, y)]_0\in G([e])$, we have $$\alpha_f \phi_{e, f}([(x, y)]_0) = [(f+ x, f+ y)]_0 = [(x, y)]_0 + [(e, f)]_0 \text{ in } G(E)$$ Since $e = e + e$ and $f = f +f$, we must have $[(e, f)]_0 = [(e + f, e+f)]_0$ in $G(E)$; that means, $[(e, f)]_0$ is the identity of $G(E)$. Hence, we obtain that $$\alpha_f \phi_{e, f}([(x, y)]_0) = [(x, y)]_0 + [(e, f)]_0 = [(x, y)]_0 = \alpha_e([(x, y)]_0) \text{ in } G(E),$$ that means, $\alpha_f \phi_{e, f} = \alpha_e$. By the Universal Property of $\varinjlim_{x\in \rm{Idem}(SP(E))} G([x])$, there is a unique group homomorphism $\beta: \varinjlim_{x\in \rm{Idem}(SP(E))} G([x])\longrightarrow G(E)$ such that $\beta \psi_e = \alpha_e$ for all $e\in \rm{Idem(SP(E))}$. Then, for each $[(x, y)]_0\in G([e])$, we have $[(x, y)]_0 = \beta([(x, y)]_0)$ in $G(E)$. Since $\text{SP}(E) =\bigsqcup_{x\in \rm{Idem}(SP(E))}[x]$, we immediately get that $\beta$ is surjective.
	
	Let $[(x, y)]_0\in G([e])$ with  $\beta([(x, y)]_0) = [(x, y)]_0 = 0_{G(E)}$. We then have $[(x, y)]_0 = [e, e]_0$ in $G(E)$, that means, $x + e + z = y + e + z$ for some $z\in \text{SP}(E)$. By item (2), there exists an idempotent $f\in\text{SP}(E)$ such that $z\in [f]$. Then, since $x + e, y +e \in [e]$ and $z + f\in [f]$, we get that $x + e +f, y +e +f, z +f +e \in [e + f]$. Since $x + e + z +f = y + e + z +f$, $ 0_{G[e +f]} =[(x + e +f, y +e +f)]_0 = \phi_{e, e+f}([(x , y)]_0)$, and so $[(x , y)]_0$ is zero in $\varinjlim_{x\in \rm{Idem}(SP(E))} G([x])$. Therefore, $\beta$ is injective, and so it is an isomorphism.
	
	
	(4) Let $x\in \rm{Idem}(SP(E))$. We claim that $x + [x]$ is a group with identity $x$. Indeed, by Item (1), $x + [x]$  is  clearly a monoid with identity $x$. Let $a$ be an arbitrary element of $[x]$. Then, by \cite[Lemma 2.10]{CGGMS}, $ n a$ is an idempotent of $[x]$ for some $n\in\mathbb{N}^+$. By item (1), we get that $x = n a$, and so $x + a + (n-1)a = x + na = x + x = x$. This shows that $x + [x]$ is a group with identity $x$, proving the claim. 
	
	Let $\iota_x: [x] \longrightarrow G([x])$ be the canonical semigroup homomorphism, that means, $\iota_x(y) = [y, y+x]_0$ for all $y\in [x]$. We note that $\phi_{x, x}: [x]\longrightarrow [x]$, $y\longmapsto x + y$, is a semigroup homomorphism with $\phi_{x, x}([x]) = x + [x]$. By the Universal Property of $G([x])$, there is a unique group homomorphism $\delta: G([x])\longrightarrow x + [x]$ such that $\delta \iota_x =\phi_{x, x}$, and so $$\delta([y, z]_0) = \delta([y, y+z]_0) - \delta([z, y+ z]_0) = \delta \iota_x(y)-\delta \iota_x(z) = x +y - (x + z)$$ for all $y, z\in [x]$. This implies that $\delta([y, z]_0) =0$ $\Longleftrightarrow x + y = x + z$ in $[x]$, and so $[(y, z)]_0 = [(x, x)]_0 = 0_{G([x])}$. Therefore, $\delta$ is injective.
	
	For every $x + y \in x + [x]$, we have $\delta([(y, x +y)]_0) = \delta \iota_x(y) = \phi_{x, x}(y) = x +y$, and so $\delta$ is surjective. Therefore, we obtain that $G([x])\cong x + [x]$.
	
	If $x = 0$, then $G([0]) = Z(\text{SP}(E)) = \text{SP}(E_{S_E}) = G(E_{S_E}) = [0]$, as desired. Suppose $x\neq 0$ and $H:= \overline{\rm{supp}(x)}$. It is obvious that $x + x=x\in \text{SP}(E_H)$. Then, since $\text{SP}(E_H)$ is an order-ideal of $\text{SP}(E)$ (by Theorem \ref{sandmonstruc} (3)), $x + [x]\subseteq \text{SP}(E_H)$. By Corollary \ref{idem-sathere-rem}, $x$ is the biggest idempotent of  $\text{SP}(E_H)$. According to Theorem \ref{sandmonstruc} (7), we immediately obtain that $G(E_H) = x +\text{SP}(E_H)$, and so 
	$x + [x]\subseteq x +\text{SP}(E_H) = G(E_H)$.  Conversely, for every $y\in G(E_H)$, by Proposition \ref{recur}, $y = x + z$ for some $z\in \text{SP}(E_H)$. By \cite[Lemma 2.10]{CGGMS}, $m z$ is an idempotent of $\text{SP}(E_H)$ for some $m\in \mathbb{N}^+$, and so $mz\le x$. This implies that $x + z \le x + x = x$, and so $x + z \in [x]$. Then, $y = x + z = x + (x + z)\in x + [x]$, that means, $G(E_H) \subseteq x + [x]$.
	
(5) Let $G$ be a maximal idempotent of $\text{SP}(E)$, that means, $G = Z(x + \text{SP}(E))$, the unit group of $x + \text{SP}(E)$, for some $x \in \rm{Idem}(SP(E))$. If $x = 0$, then $G = Z(\text{SP}(E)) = \text{SP}(E_{S_E}) = G(E_{S_E}) = [0]$. Consider the case when $x\neq 0$. Let $H= \overline{\rm{supp}(x)}$. By Corollary \ref{idem-sathere-rem}, $x$ is the biggest idempotent of  $\text{SP}(E_H)$. By the proof of Theorem \ref{sandmonstruc} (7), 
$G = Z(x+\text{SP}(E)) = x + \text{SP}(E_H) = G(E_H)$. By Item (4), we have $G = G(E_H) = x + [x]$. Using this observation, Item (4) and 
Theorem \ref{sandmonstruc} (7), we immediately obtain that the set of all maximal subgroups of $\rm{SP}$$(E)$ is exactly the set $\{x + [x]\mid x\in \rm{Idem(SP(E))}\},$ thus finishing the proof.
\end{proof}

\section{A connection between sandpile monoids and weighted Leavitt path algebras} \label{weightedgraphssec4}
In this section, based on Theorem \ref{idem-fiter-here-ordideal},
we show that the lattice of all idempotents of the sandpile momoid $\text{SP}(E)$ of a sandpile graph $E$ is both isomorphic to the lattice of all nonzero graded ideals of the Leavitt path algebra $L_{\K}(E)$ and the lattice of all ideals the weighted Leavitt path algebra $L_{\K}(E, \omega)$ generated by vertices (Theorem \ref{classideal}). Also, we give the structure of the Leavitt path algebra $L_{\K}(E)$ of a sandpile graph $E$, extending Zelmanov et al.'s result \cite[Theorem 1]{Zel} to the sandpile graph setting, and completely describe the structure of $L_{\K}(E)$ such that the lattice of all idempotents of $\text{SP}(E)$ is a chain (Theorem \ref{structheo1}). Consequently, we completely describe the structure of the (weighted) Leavitt path algebra of a sandpile graph $E$ such that $\text{SP}(E)$ has exactly two idempotents (Theorem \ref{twoidem}).\medskip

We begin this section by recalling some basic notions of weighted graphs.
 
\begin{deff}\label{weightedgraphdef}
A \emph{weighted graph} is a pair $(E,\omega)$, where $E$ is a graph and $\omega: E^1\rightarrow \mathbb N^+$ is a map. If $e\in E^1$, then $\omega(e)$ is called the \emph{weight} of $e$. A weighted graph $(E, \omega)$ is called  \emph{row-finite} if the graph $E$ is row-finite. A weighted graph $(E, \omega)$ is called  \emph{finite} if the graph $E$ is finite.

For each regular vertex $v$ in a weighted graph $(E,\omega)$ we set $\omega(v):=\max\{\omega(e)\mid e\in s^{-1}(v)\}$.  This gives  a map (called $\omega$ again) $\omega:E_{\reg}^0\rightarrow \mathbb N^+$. 

A row-finite weighted graph $(E, \omega)$ is called a \emph{vertex weighted graph} if for any $v\in E^0_{\reg}$, $\omega(e) = \omega(e')$ for all  $e,e' \in s^{-1}(v)$.     

A vertex weighted graph is called a \emph{balanced weighted graph} if  $\omega(v)=|s^{-1}(v)|$ for all $v\in E_{\reg}^0$.   
\end{deff}

It is worth mentioning the following note.

\begin{remark}\label{balancedweightedgraphremark}
Note that any row-finite directed graph $E$ can be given the structure of a balanced weighted graph $(E, \omega)$ by assigning $\omega(v)=|s^{-1}(v)|$ for all $v\in E_{\reg}^0$, and the now-so called {\it balanced weighted graph associated with} $E$.  
\end{remark}


We next recall the notion of weighted Leavitt path algebras. These are algebras associated to row-finite weighted graphs. We refer the reader to \cite{H} and \cite{Pre} for a detailed analysis of these algebras. 
\begin{deff}\label{weighteddef}
Let $(E,\omega)$ be a  row-finite weighted graph and $\K$ a field.  The free $\K$-algebra generated by $\{v,e_i,e_i^*\mid v\in E^0, e\in E^1, 1\leq i\leq \omega(e)\}$ subject to relations
\begin{enumerate}[(i)]
		\item $uv=\delta_{uv}u,  \text{ where } u,v\in E^0$,
		\medskip
		\item $s(e)e_i=e_i=e_ir(e),~r(e)e_i^*=e_i^*=e_i^*s(e),  \text{ where } e\in E^1, 1\leq i\leq \omega(e)$,
		\medskip
		\item 
		$\sum_{e\in s^{-1}(v)}e_ie_j^*= \delta_{ij}v, \text{ where } v\in E_{\reg}^0 \text{ and } 1\leq i, j\leq \omega(v)$, 
		\medskip 
		\item $\sum_{1\leq i\leq \omega(v)}e_i^*f_i= \delta_{ef}r(e), \text{ where } v\in E_{\reg}^0 \text{ and } e,f\in s^{-1}(v)$,
\end{enumerate}
is called the {\it weighted Leavitt path algebra} of $(E,\omega)$ over $\K$, and denoted $L_\K(E,\omega)$, where $\delta$ is the Kronecker delta. In relations (iii) and (iv) we set $e_i$ and $e_i^*$ to be zero whenever $i > \omega(e)$.    
\end{deff}

Note that if the weight of each edge is $1$, then $L_\K(E,\omega)$ reduces to the usual Leavitt path algebra $L_\K(E)$  of the graph $E$.  It is easy to see that the mappings given by $v\longmapsto v$ for all $v\in E^0$, $e_i\longmapsto e^*_i$ and $e^*_i\longmapsto e_i$ for all $e\in E^1$ and $1\le i\le \omega(e)$, produce an involution on the $\K$-algebra $L_{\K}(E, \omega)$. Also, $L_{\K}(E, \omega)$ is a $\mathbb{Z}$-graded $\K$-algebra with grading induced by $\deg(v) =0$ for all $v\in E^0$, $\deg(e_i) =1$ and $\deg(e^*_i) = -1$ for all $e\in E^1$ and $1\le i\le \omega(e)$.\medskip


Let $\K$ be a field, $E$ a sandpile graph and $(E, \omega)$ the 
balanced weighted graph associated with $E$. Let $L_{\K}(E)$ be the Leavitt path algebra of $E$ and $L_{\K}(E, \omega)$ the weighted Leavitt path algebra of $(E, \omega)$. We denote by $\mathcal{L}_{\text{gr}}(L_{\K}(E))$ the set of nonzero graded ideals of $L_{\K}(E)$, and denote by $\mathcal{L}_{\text{ver}}(L_{\K}(E, \omega))$ the set of all ideals of $L_{\K}(E, \omega)$ generated by vertices. It is well-known that every nonzero graded ideal $I$ of $L_{\K}(E)$ always contains a vertex $v\in E^0$. Since $E$ is a sandpile graph, there is a path $p$ in $E$ such that $v = s(p)$ and $s = r(p)$, and so  $s = r(p) = p^*p = p^*vp\in I$, that means, $I$ always contains the sink $s$. This implies that $(\mathcal{L}_{\text{gr}}(L_{\K}(E)), \subseteq)$ is a lattice. Similarly, we have that $(\mathcal{L}_{\text{ver}}(L_{\K}(E, \omega)), \subseteq)$ is a lattice. Moreover, the following theorem shows that these lattices are isomorphic to the lattice of all idempotents of $\text{SP}(E)$.\medskip

\begin{theorem}\label{classideal}
Let $\K$ be a field, $E$ a sandpile graph and  $(E, \omega)$ the balanced weighted graph associated with $E$. Then, there is a lattice isomorphism between the following lattices:
	
$(1)$ The lattice $\rm{Idem}(SP(E))$ of all idempotents of $\rm{SP}$$(E)$;
	
$(2)$ The lattice $\mathcal{F}_E$ of all filters of $\mathcal{C}_E$;
	
$(3)$ The lattice $\mathcal{H}_E$ of all nonempty saturated hereditary subsets of $E^0$;
	
$(4)$ The lattice $\mathcal{L}(\rm{SP}$$(E))$ of all order-ideals of $\rm{SP}$$(E)$;
	
$(5)$ The lattice $\mathcal{L}_{\rm{gr}}(L_{\K}(E))$ of all nonzero graded ideals of $L_{\K}(E)$;
	
$(6)$ The lattice $\mathcal{L}_{\rm{ver}}(L_{\K}(E, \omega))$ of all ideals of $L_{\K}(E, \omega)$ generated by vertices.
\end{theorem}
\begin{proof}
By Theorem \ref{idem-fiter-here-ordideal}, we obtain that \[\rm{Idem}(SP(E))\cong \mathcal{F}_E\cong \mathcal{H}_E\cong \mathcal{L}(\text{SP}(E))\] as lattices. By \cite[Theorem 2.5.8]{TheBook}, we have $$\mathcal{H}_E\cong \mathcal{L}_{\text{gr}}(L_{\K}(E))$$ as lattices.

We next claim that \[\mathcal{H}_E\cong \mathcal{L}_{\text{ver}}(L_{\K}(E, \omega))\] as lattices. Indeed, let $I$ be an ideal of $L_{\K}(E, \omega)$ generated by vertices. It is obvious that $I \cap E^0\neq \emptyset$.
We show that $I\cap E^0$ is a hereditary and saturated subset of $E^0$. Let $e\in E^1$ with $v:=s(e)\in I$. We then have $e_i = ve_i\in I$ for all  $1\le i\le \omega(e)$, and so $r(e)=\sum_{1\leq i\leq \omega(v)}e_i^*e_i\in I$. This implies that $I\cap E^0$ is hereditary. Let $v\in E_{\reg}^0$ with the property that $r(s^{-1}(v))\subseteq I\cap E^0$. We then have $r(e_1)\in I$ for all $e\in s^{-1}(v)$, and so $e_1 = e_1 r(e_1)\in I$ for all $e\in s^{-1}(v)$. Therefore, we obtain that $v = \sum_{e\in s^{-1}(v)}e_1e_1^*\in I$, and so $I\cap E^0$ is saturated, as desired.

On the other hand if $H$ is hereditary and saturated and $I(H)$ is an ideal of $L(E,\omega)$ generated by $H$, then $I(H) \cap E^0= H$ by \cite[Theorem 2.10]{HN}. From these observations, we immediately obtain that the corresponding $\alpha: \mathcal{L}_{\text{ver}}(L_{\K}(E, \omega))\longrightarrow \mathcal{H}_E$, defined by $I\longmapsto I\cap E^0$, is a lattice isomorphism, thus finishing the proof.
\end{proof}

As a corollary of Theorem \ref{classideal}, we have the following useful fact.

\begin{cor}\label{chain}
Let $\K$ be a field, $E$ a sandpile graph and  $(E, \omega)$ the balanced weighted graph associated with $E$. Then, the following statements are equivalent:
	
$(1)$ The lattice $\rm{Idem}(SP(E))$ is a chain;
	
$(2)$ The lattice $\mathcal{F}_E$ is a chain;
	
$(3)$ The partially ordered set $\mathcal{C}_E$	is a chain;
	
$(4)$ The lattice $\mathcal{H}_E$ is a chain;
	
$(5)$ The lattice $\mathcal{L}_{\rm{gr}}(L_{\K}(E))$ is a chain;
	
$(6)$ The lattice $\mathcal{L}_{\rm{ver}}(L_{\K}(E, \omega))$ is a chain. 
\end{cor}
\begin{proof}
By Theorem \ref{classideal}, it suffices to show the equivalence of (2) and (3). It is obvious that if the partially ordered set $\mathcal{C}_E$	is a chain, then so is the lattice $\mathcal{F}_E$. Assume that the lattice $\mathcal{F}_E$ is a chain, but $\mathcal{C}_E$	is not a chain. Then, there exist strongly connected cyclic components $C_1$ and $C_2$ of $E$ such that $C_1\nleq C_2$ and $C_2\nleq C_1$. Let $F_{C_1} = \{C\in \mathcal{C}_E\mid C_1\le C\}$ and $F_{C_2} = \{D\in \mathcal{C}_E\mid C_2\le D\}$. It is obvious that $F_{C_1}$ and $F_{C_2}\in \mathcal{F}_E$. Since $C_1\nleq C_2$ and $C_2\nleq C_1$, we have $F_{C_1}\nleq F_{C_2}$ and $F_{C_2}\nleq F_{C_1}$, that means, $\mathcal{F}_E$ is not a chain, a contradiction. Therefore, $\mathcal{C}_E$	is a chain, thus finishing the proof.
\end{proof}	

In the remainder of this section, we provide structural connections between sandpile monoids and (weighted) Leavitt path algebras. To do so, we need to represent useful notions and facts.

\begin{deff}[{\cite[Definition 2.5.16]{TheBook}}]\label{hedgehoggraph}
	Let $E$ be a graph and let $H$ be a hereditary subset of $E^0$. Consider the set
	\[F(H) =\{\alpha \mid \alpha = e_1e_2\hdots e_n, s_E(e_n)\notin H, r_E(e_n)\in H\}.\] Let
	$\overline{F}(H)$ be another copy of $F(H)$ and we write $\overline{\alpha}$ for the
	copy of $\alpha$ in $\overline{F}(H)$. Define a graph $E(H)$ as follows:
	\begin{equation*}
	\begin{array}{l}
	E(H)^0 = H\cup F(H)\\
	E(H)^1 = s^{-1}_E(H)\cup \overline{F}(H)\\
	\end{array}
	\end{equation*}
	and extend $s_E$ and $r_E$ to $E(H)$ by defining
	$s_{E(H)}(\overline{\alpha}) = \alpha$ and
	$r_{E(H)}(\overline{\alpha}) = r(\alpha)$.
\end{deff}

For clarification, we illustrate Definition \ref{hedgehoggraph} by presenting the following example.

\begin{example}
	Let $E$ be the graph
	$$
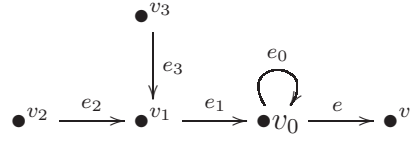
\xymatrix{&\bullet^{v_3} \ar[d]^{e_3}&&\\
		\bullet^{v_2} \ar[r]^{e_2}& \bullet^{v_1} \ar[r]^{e_1} &
		\bullet{v_0} \ar@(ul,ur)^{e_0} \ar[r]^e& \bullet^v}$$ and
	$H=\{v_0, v\}$. Then $F(H)=\{e_1,e_2e_1,e_3e_1\}$. Therefore, the graph
	$$
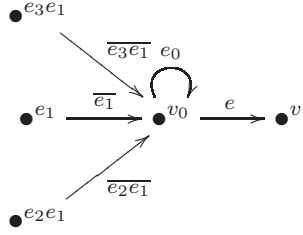
\xymatrix{\bullet^{e_3e_1} \ar[dr]^{\overline{e_3e_1}}&& \\
		\bullet^{e_1} \ar[r]^{\overline{e_1}}& \bullet^{v_0} \ar@(ul,ur)^{e_0} \ar[r]^e & \bullet^v\\
		\bullet^{e_2e_1} \ar[ur]_{\overline{e_2e_1}}&& }$$ represents the graph $E(H)$.
\end{example}

We should mention the following important result.

\begin{theorem}[{\cite[Theorem 2.5.19]{TheBook}}]\label{hedgehog graphtheo}
	Let $\K$ be a field, $E$ an arbitrary graph, and $H$ a hereditary subset of $E^0$. Let $I(H)$ be the ideal of $L_{\K}(E)$ generated by $H$. Then, $I(H)\cong L_{\K}(E(H))$ as $\K$-algebras.
\end{theorem}

The following lemma is useful for describing the structure of Leavitt path algebras of sandpile graphs.

\begin{lemma}\label{hedgehoggraph-simple}
	Let $\K$ be a field, $E$ an arbitrary graph and $H$ a hereditary subset of $E^0$. Then, the following statements hold:
	
	$(1)$ $L_{\K}(E_H)$ is a $\K$-subalgebra of $L_{\K}(E(H))$;
	
	$(2)$ If $H$ is simply a cycle based at $v$, then $L_{\K}(E(H))\cong \M_{\Lambda_v}(\K[x, x^{-1}])$, where $\Lambda_v$ is the set of all paths in $E$ which end at $v$, but which do
	not contain all the edges of $H$;
	
	$(3)$ If $L_{\K}(E_H)$ is purely infinite simple, then $L_{\K}(E(H))$ is also purely infinite simple.
\end{lemma}
\begin{proof}
(1) By the graded uniqueness theorem (see, e.g., \cite[Theorem 2.2.15]{TheBook}), the map $\phi: L_{\K}(E_H)\longrightarrow L_{\K}(E(H))$, defined by: $v\longmapsto v$ for all $v\in H$ and $e\longmapsto e$, $e^*\longmapsto e^*$ for all $e\in s^{-1}(H)$, is an injective homomorphism of $\K$-algebras, and so $L_{\K}(E_H)$ can be considered as a $\K$-subalgebra of $L_{\K}(E(H))$.
	
(2) It immediately follows from Theorem \ref{hedgehog graphtheo} and \cite[Lemma 2.7.1]{TheBook}.
	
(3) Assume that $L_{\K}(E_H)$ is purely infinite simple. Then, by \cite[Theorem 11]{aap:pislpa}, $E_H$ has a cycle, and so $E(H)$ has a cycle, since $E(H)$ contains $E_H$ as a subgraph. Note that if $c$ is a cycle in $E(H)$, then $c$ lies in $E_H$. Since  $L_{\K}(E_H)$ is purely infinite simple and by \cite[Theorem 11]{aap:pislpa}, every cycle in $E_H$ has an exit. This implies that every cycle in $E(H)$ has an exit.
	
Let $G$ be a nonempty saturated hereditary subset of $E(H)^0$. If there exists an element $\alpha\in F(H)$ such that $\alpha\in G$, then $r(\alpha)\in H\cap G$. Otherwise, since $H\neq \varnothing$, $G$ must contain a vertex which lies in $H$. In any case, $G$ contains a vertex $v\in H$. Let $\overline{\{v\}}$ be the saturated and hereditary closure of $\{v\}$ in $E_H$. Since $L_{\K}(E_H)$ is purely infinite simple and by \cite[Theorem 11]{aap:pislpa}, $H = \overline{\{v\}}$, and so $H \subseteq G$. Then, since every $\alpha\in F(H)$ is a regular vertex in $E(H)$ and connects to $H$ via the edge $\overline{\alpha}$, it lies in $G$ (since $G$ is saturated). This shows that $G = (E(H))^0$, that means, the only hereditary and saturated subsets of $(E(H))^0$ are $\varnothing$ and $(E(H))^0$. From these observations and by \cite[Theorem 11]{aap:pislpa}, we obtain that $L_{\K}(E(H))$ is purely infinite simple, thus finishing the proof.
\end{proof}	

Let $E$ be an arbitrary graph and $\mathcal C_E$  the  partially ordered  set of all strongly connected cyclic components of $E$. A sequence of distinct elements $C_1, \hdots, C_k$ of $\mathcal C_E$ is a {\it chain of length} $k$ if $C_1\le \cdots\le C_k$. An edge $e$ is called an {\it exit} to a strongly connected cyclic component $C$ of $E$ if  $s(e)\in C^0$ and $r(e)\notin C^0$, where $C^0$ is the set of all vertices in $E$ which lie on $C$.

For a hereditary subset $H$ of $E^0$, we define the {\it quotient graph} $E/H$ as follows:
$$(E/ H)^0=E^0\setminus H   \ \ \mbox{and} \ \  (E/H)^1=\{e\in E^1\;| \ r(e)\notin H\}.$$ The source and range maps of $E/H$ are the source and range maps  restricted from $E$.

We are now in position to present the second main result of this section describing the structure of Leavitt path algebras of sandpile graphs, which extends Zelmanov et al.'s result \cite[Theorem 1]{Zel} to the sandpile graph setting.

\begin{theorem} \label{structheo1}
Let $\K$ be a field, $E$ a sandpile graph and let $t$  be the maximal length of a chain of elements in $\mathcal{C}_E$. Then  $L_{\K}(E)$ has a finite chain of graded ideals, $0 < I_0 < I_1 < \cdots < I_t = L_{\K}(E)$, satisfying the following conditions:

$(1)$ $I_0 = \rm{Soc}$$(L_{\K}(E)) \cong \M_{\Lambda_s}(\K)$, where $\Lambda_s$ is the set of paths in $E$ ending at $s$, 

$(2)$ $I_{i+1}/I_i$ $(0 \le i < t)$ is a finite sum of purely infinite simple Leavitt path $\K$-algebras and  matrix algebras $\M_{\Lambda}(\K[x, x^{-1}])$, where $\Lambda$ is a nonempty set,

$(3)$ $I_t/I_{t-1}$ is a finite sum of unital purely infinite simple Leavitt path $\K$-algebras and matrix algebras $\M_{n}(\K[x, x^{-1}])$, where $n\in \mathbb{N}^+$;

$(4)$ The ideals $I_i$ are invariant under every graded automorphism of $L_{\K}(E)$.

Consequently, the lattice $\rm{Idem}(SP(E))$ is a chain if and only if $I_{i+1}/I_i$ is either a  purely infinite simple Leavitt path $\K$-algebra or a matrix algebra $\M_{\Lambda}(\K[x, x^{-1}])$ for all $0 \le i \le t-1$.
\end{theorem}
\begin{proof}
Let $I_0$ be the ideal of $L_{\K}(E)$ generated by $s$. Let $S_E$ be the hereditary saturated subset of $E^0$ including all vertices which do not connect to any cycle. It is obvious that the hereditary saturated closure of $\{s\}$ is $S_E$, and so $I_0 = I(S_E)$, the ideal of $L_{\K}(E)$ generated by $S_E$. By \cite[Corollary 2.6.5]{TheBook}, $I_0 \cong \M_{\Lambda_s}(\K)$, where $\Lambda_s$ is the set of paths in $E$ ending at $s$. By \cite[Theorem 5.2]{AMMS}, $I_0= I(S_E)$ is exactly the socle of the Leavitt path algebra $L_{\K}(E)$. By \cite[Theorem 2.4.15]{TheBook}, $L_{\K}(E)/I_0 \cong L_{\K}(E/S_E)$. Let $\{C_{11}, C_{12}, \hdots, C_{1k_1}\}$ be the set of all strongly connected cycle components of $E/S_E$ without exits. We denote by $C^0_{1i}$ the set of all vertices which lie on $C_{1i}$. Let $I_1$ be the ideal of $L_{\K}(E)$ generated by $\cup^{k_1}_{i=1}C^0_{1i}$. Since every vertex in $E$ connects to $s$, we must have $I_0 < I_1$. Let $H_1$ be the hereditary saturated closure of $\cup^{k_1}_{i=1}C^0_{1i}$. By Theorem \ref{hedgehog graphtheo}, $I_1 = I(H_1)\cong L_{\K}(E(H_1))$, and so $I_1/I_0\cong \oplus^{k_1}_{i=1} I^i_1$, where $I^i_1$ is the ideal of $L_{\K}(E/S_E)$ generated by $C^0_{1i}$. Clearly, each $C^0_{1i}$ is a hereditary subset of $E^0\setminus S_E$. By Theorem \ref{hedgehog graphtheo}, $I_1/I_0\cong \oplus^{k_1}_{i=1} L_{\K}(E(C^0_{1i}))$.
If $C_{1i}$ is simply a cycle, then $L_{\K}(E(C^0_{1i}))\cong \M_{\Lambda_v}(\K[x, x^{-1}])$, where $\Lambda_v$ is the set of all paths in $E$ which end at $v:= s(C_{1i})$, but which do not contain all the edges of $C_{1i}$, by Lemma \ref{hedgehoggraph-simple} (2). If $C_{1i}$ has at least two distinct cycles, then we claim that $L_{\K}(C_{1i})$ is purely infinite simple. Indeed, since $C_{1i}$ is strongly connected, $C_{1i}$ has only the trivial hereditary and saturated subsets. Since $C_{1i}$ has both at least two distinct cycles and is strongly connected, every cycle in $C_{1i}$ has an exit. By \cite[Theorem 3.1.10]{TheBook}, $L_\K(C_{1i})$ is purely infinite simple, showing the claim. Then, by Lemma \ref{hedgehoggraph-simple} (3), we immediately obtain that $L_{\K}(E(C^0_{1i}))$ is purely infinite simple.
	
By \cite[Theorem 2.4.15]{TheBook}, $L_{\K}(E)/I_1 \cong L_{\K}(E/H_1)$. Let $\{C_{21}, C_{22}, \hdots, C_{2k_2}\}$ be the set of all strongly connected cycle components of $E/H_1$ without exits. Let $I_2$ be the ideal of $L_{\K}(E)$ generated by $\cup^{k_2}_{i=1}C^0_{2i}$. Similar to the above argument, we have 
$I_1 < I_2$, $I_2 \cong L_{\K}(E(H_2))$ and $I_2/I_1\cong \oplus^{k_2}_{i=1} L_{\K}(E(C^0_{2i}))$, where $L_{\K}(E(C^0_{2i}))$ is either a purely infinite simple Leavitt path $\K$-algebras or isomorphic to $\M_{\Lambda}(\K[x, x^{-1}])$. We continue to repeat this process $t+1$ times. We then obtain that  $L_{\K}(E)$ has a finite chain of graded ideals, $0 < I_0 < I_1 < \cdots < I_t = L_{\K}(E)$, such that $I_{i+1}/I_i$ is a finite sum of purely infinite simple Leavitt path $\K$-algebras and matrix algebras $\M_{\Lambda}(\K[x, x^{-1}])$, where $\Lambda$ is an infinite set, for $0\le i\le t-2$. For $i = t-1$, we have $I_t/I_{t-1}\cong \oplus^{k_t}_{i=1} L_{\K}(E(C^0_{ti}))$, where $\{C_{t1}, \hdots , C_{tk_t}\}$ is the set of all minimal elements of the partially ordered set $(\mathcal{C}_E, \le)$. Hence, $F(C^0_{ti})$ is finite for all $1\le i\le k_t$. This implies that $L_{\K}(E(C^0_{ti}))$ is either a unital purely infinite simple Leavitt path $\K$-algebras or isomorphic to $\M_{n}(\K[x, x^{-1}])$ for some $n\in \mathbb{N}^+$.

Let $\alpha$ be a graded automorphism of $L_{\K}(E)$. Then, since $I_0 = \rm{Soc}$$(L_{\K}(E))$, $I_0 = \alpha(I_0)$. We next claim that $I_1= \alpha(I_1)$. Indeed, since $I_0 < I_1$, $I_0 = \alpha(I_0) < \alpha(I_1)$. Since $I_1$ is a graded ideal of $L_{\K}(E)$, so is $\alpha(I_1)$. By \cite[Theorem 2.5.8]{TheBook}, $\alpha(I_1)$ is equal to the ideal $I(H)$ of $L_{\K}(E)$ generated by a hereditary and saturated subset $H$ of $E^0$. By Theorem \ref{hedgehog graphtheo}, $\alpha(I_1)\cong L_{\K}(E(H))$, and so $\alpha(I_1)/I_0\cong L_{\K}(E(H)/S_E)$. Since $\alpha$ is a graded automorphism of $L_{\K}(E)$, $\alpha(I_1)/I_0\cong I_1/I_0$ as $\mathbb{Z}$-graded $\K$-algebras. Therefore,  $\alpha(I_1)/I_0\cong \oplus^{k_1}_{i=1} L_{\K}(E(C^0_{1i}))$. Assume that $C^0_{ij} \subset H$ for some $2\le i\le t$ and $1\le j\le k_i$. Then, $C_{ij}$ connects to $C_{1h}$ for some $1\le h\le k_1$. Let $\overline{C^0_{1h}}$ be the hereditary and saturated subset of $(E(H)/S_E)^0 = E(H)^0\setminus S_E$. Let $c$ be an arbitrary cycle in $C_{ij}$, and $p$ the infinite path $c^{\infty} = c c \cdots c\cdots$ in $C_{ij}$. Clearly, the set $c^0$ of all vertices which lie on $c$ is contained in $(E(H)/S_E)^0\setminus \overline{C^0_{1h}}$, and the set of all paths $x$ such that
$s(x)\in c^0$ and $r(x)\in \overline{C^0_{1h}}$ is different from the empty set. By \cite[Theorem 5.2]{Clark2017}, $\alpha(I_1)/I_0\cong L_{\K}(E(H)/S_E)$ is indecomposable. This implies that $k_1 =1$, that means, $\alpha(I_1)/I_0$ is either a purely infinite simple Leavitt path $\K$-algebras or isomorphic to $\M_{\Lambda}(\K[x, x^{-1}])$. Hence, $\alpha(I_1)/I_0$ is graded simple. On the other hand, $\alpha(I_1)/I_0\cong L_{\K}(E(H)/S_E)$ always contains a proper graded ideal $I(\overline{C^0_{1h}})$, a contradiction. Therefore, we obtain that $C^0_{ij} \nsubseteq H$ for all $2\le i\le t$ and $1\le j\le k_i$; equivalently, $C^0_{ij} \cap H = \emptyset$ for all $2\le i\le t$ and $1\le j\le k_i$ (since $H$ is a hereditary subset of $E^0$). Since $\alpha(I_1)/I_0\neq 0$, $C^0_{1i} \subset H$ for some $1\le i\le k_1$. We note that every strongly connected cycle component of $E(H)$ is exactly a strongly connected cycle component of $E$. Therefore, we have $$\alpha(I_1)/I_0\cong L_{\K}(E(H)/S_E) \cong \bigoplus_{j\in J} L_{\K}(E(C^0_{1j}))$$ for some nonempty subset $J \subseteq \{1, 2, \hdots, k_1\}$, and so 
$$\bigoplus_{j\in J} L_{\K}(E(C^0_{1j}))\cong \alpha(I_1)/I_0\cong I_1/I_0\cong \bigoplus^{k_1}_{i=1} L_{\K}(E(C^0_{1i})).$$
Then, since each $L_{\K}(E(C^0_{1i}))$ is indecomposable and by \cite[Lemma 3.8]{LamBook}, we must have $J = \{1, 2, \hdots, k_1\}$, that means, $C^0_{1i} \subset H$ for all $1\le i\le k_1$. This shows that $\alpha(I_1) = I(H) = I_1$, proving the claim. Similar to the above argument, we obtain  that $\alpha(I_i) = I_i$ for all $0\le i \le t$. 

Assume that $\rm{Idem}(SP(E))$ is a chain. Then, by Corollary \ref{chain}, the partially ordered set $\mathcal{C}_E$	is also a chain. By Items (2) and (3), we immediately obtain that $I_{i+1}/I_i$ is either a  purely infinite simple Leavitt path $\K$-algebra or a matrix algebra $\M_{\Lambda}(\K[x, x^{-1}])$ for all $0 \le i \le t-1$. Conversely, by the construction of all the ideals $I_i$, the partially ordered set $(\mathcal{C}_E, \le)$ is a chain. By Corollary \ref{chain}, the lattice $\rm{Idem}(SP(E))$ is a chain,  thus finishing the proof.
\end{proof}	

For clarification, we illustrate Theorem \ref{structheo1} by presenting the following example.

\begin{example}
Let $K$ be a field and $E$ the following pictured graph
	$$ 
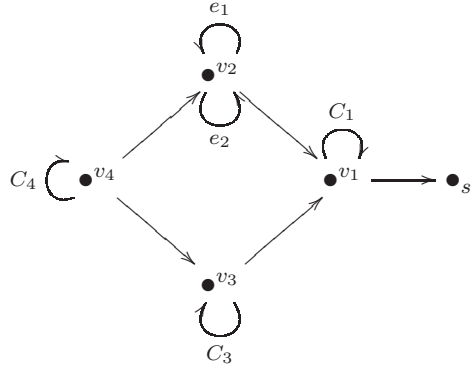
\xymatrix{&\bullet^{v_2}\ar[rd]\ar@(ur,ul)_{e_1}\ar@(dl,dr)_{e_2}&\\
		\bullet^{v_4}\ar@(ld,lu)^{C_4}\ar[ru]\ar[rd]&&\bullet^{v_1}\ar[r]\ar@(ul,ur)^{C_1}&\bullet_s\\
		&\bullet^{v_3}\ar[ru]\ar@(dr,dl)^{C_3}&}$$
Then we have $\mathcal{C}_E = \{C_1, C_2, C_3, C_4\}$, where $C_2 = \xymatrix{\bullet^{v_2}\ar@(ur,ul)_{e_1}\ar@(dl,dr)_{e_2}}$, and the chain $C_1 \le C_2\le C_4$ is a chain of maximal length in $\mathcal{C}_E$. Then, the chain of graded ideals of $L_K(E)$, $I_0 < I_1 < I_2 < I_3$, introduced in Theorem \ref{structheo1}, is constructed as follows: $I_0 = I(\{s\})$, $I_1 =I(\{v_1\}) = I(\{v_1, s\})$, $I_2 =I(\{v_3, v_2\}) = I(\{v_3, v_2, v_1, s\})$ and 
$I_3 =I(\{v_4\}) = I(\{v_4, v_3, v_2, v_1, s\})$, where $I(H)$
is the ideal of $L_K(E)$ generated by $H$. We have that $I_1/I_0$ is the ideal of $L_K(E/S_E)$ generated by $v_1$, where $S_E =\{s\}$, and so $I_1/I_0\cong \M_{\Lambda_1}(K[x, x^{-1}])$, where $\Lambda_1$ is the set of all paths in $E$ ending at $v_1$.  Also, $I_2/I_1$ is  is the ideal of $L_K(E/H_1)$ generated by $\{v_2, v_3\}$, where $H_1 =\{v_1, s\}$, and so $I_2/I_1\cong L_K(E(\{v_2\})) \oplus L_K(E(\{v_3\}))$, $L_K(E(\{v_2\})) \cong \M_{\Lambda_2}(L_K(C_2))$ is purely infinite simple, and $L_K(E(\{v_3\}))\cong  \M_{\Lambda_3}(K[x, x^{-1}])$, where $\Lambda_i$ is the set of all paths in $E$ ending at $v_i$. Finally, $I_3/I_2$ is  is the ideal of $L_K(E/H_2)$ generated by $\{v_4\}$, where $H_2 =\{v_3, v_2, v_1, s\}$, and so $I_3/I_2\cong K[x, x^{-1}]$.
\end{example}

The weighted Leavitt path algebra $L_{\K}(E, \omega)$ of a weighted graph $(E, \omega)$ over a field $\K$ is {\it vertex-simple} if $\mathcal{L}_{\rm{ver}}(L_{\K}(E, \omega)) = \{L_{\K}(E, \omega)\}$.

Let $E$ be a sandpile graph, $S_E$ the hereditary and saturated subset of $E$ which does not connect to any cycle and $(E, \omega)$ the  balanced weighted graph associated with $E$. We write $(E/S_E, \omega_r)$ to indicate that the weight function being considered on $E/S_E$ is the restriction function of the weight function $\omega$ to $E/S_E$. 

We close the article by describing completely the structure of the (weighted) Leavitt path algebra of a sandpile graph $E$ such that the sandpile monoid of $E$ has exactly two idempotents.

\begin{theorem}\label{twoidem}
Let $\K$ be a field, $E$ a sandpile graph and $S_E$ be the set of vertices in $E$ which does not connect to any cycle. Then the following statements are equivalent:
	
$(1)$ $\SP(E)$ has exactly  two idempotents;

$(2)$ The Leavitt path algebra $L_\K(E/S_E)$ is either isomorphic to $\M_n(K[x,x^{-1}])$ for some $n$ or purely infinite simple; 
	
$(3)$ The weighted Leavitt path algebra $L_\K(E/S_E, \omega_r)$ is vertex-simple.
\end{theorem}

\begin{proof}
(1) $\Longrightarrow$ (2). Assume that $\SP(E)$ has only two idempotents. By Theorem \ref{classideal}, $E$ has only one strongly connected cyclic component $\mathcal{C}$. By Theorem \ref{structheo1}, we obtain that $L_\K(E/S_E)$ is either isomorphic to $\M_n(K[x,x^{-1}])$ for some $n$ or purely infinite simple, as desired.

(2) $\Longrightarrow$ (3). By our hypothesis and \cite[Proposition 3.1.14 ]{TheBook}, the Leavitt path algebra $L_{\K}(E/S_E)$ is graded simple, and so $\mathcal{H}_E =\{E, S_E\}$. Then, by Theorem \ref{classideal}, the lattice $\mathcal{L}_{\text{ver}}(L_{\K}(E, \omega)) = \{I(S_E), L_{\K}(E, \omega)\}$, where $I(S_E)$ is the ideal of $L_{\K}(E, \omega)$ generated by $S_E$. By \cite[Theorem 2.10]{HN}, $L_{\K}(E/S_E, \omega_r)\cong L_{\K}(E, \omega)/I(S_E)$.
It follows that $L_\K(E/S_E, \omega_r)$ is vertex-simple.
	
(3) $\Longrightarrow$ (1). Assume that $L_\K(E/S_E, \omega_r)$ is vertex-simple. We should note that $I(S_E)$ is the smallest ideal of $L_{\K}(E, \omega)$ which is generated by vertices. By \cite[Theorem 2.10]{HN}, $L_{\K}(E, \omega)/I(S_E)\cong L_{\K}(E/S_E, \omega_r)$. Therefore, $\mathcal{L}_{\text{ver}}(L_{\K}(E, \omega)) =\{L_{\K}(E, \omega), I(S_E)\}$. By Theorem \ref{classideal},  $\SP(E)$ has exactly two idempotents, thus finishing the proof.
\end{proof}

\section{Acknowledgements}
The first author acknowledges Australian Research Council grant DP230103184. The second author was supported by the Vietnam Academy of Science and Technology under grant
CTTH00.01/24-25. 





\bigskip
\bigskip




\begin{thebibliography}{99}
\bibitem{aap05} G. Abrams, G. Aranda Pino, The Leavitt path algebra of a graph, {\it J. Algebra} {\bf 293} (2005), 319--334.

\bibitem{aap:pislpa} G. Abrams, G. Aranda Pino, Purely infinite simple Leavitt path algebras, {\it J. Pure Appl. Algebra} {\bf 207} (2006)  553--563.

\bibitem{TheBook}   G. Abrams, P. Ara, M. Siles Molina,  {\it Leavitt Path Algebras}, Lecture Notes in Mathematics series, Vol. 2191, Springer-Verlag Inc., 2017.




\bibitem{GeneRooz} G. Abrams, R. Hazrat, Connections between Abelian sandpile models and the K-theory of weighted Leavitt path algebras,  {\it Eur. J. Math.} {\bf 9} (2023), no. 2, Paper No. 21, 28 pp.







\bibitem{Zel} A. Alahmedi, H.Alsulami, S. Jain, E. I. Zelmanov, Structure of Leavitt path algebras of polynomial growth, {\it Proc. Natl. Acad. Sci. USA} {\bf 110} (2013), no. 38, 15222-15224.



\bibitem{amp} P. Ara, M.A. Moreno, E. Pardo, Nonstable $K$-theory for graph algebras, {\it Algebr. Represent. Theory} {\bf 10} (2007), 157--178.


\bibitem{AMMS} G. Aranda Pino, D. Mart\'{\i}n Barquero, C. Mart\'{\i}n Gonz\'{a}lez,  M. Siles Molina, Socle theory for Leavitt path algebras of arbitrary graphs, {\it Rev. Mat. Iberoam.} {\bf 26} (2010), 611-638.
 
\bibitem{BT} L. Babai,   E. Toumpakari,  \emph{A structure theory of the sandpile monoid for directed graphs}, 2010.   preprint. 
 
\bibitem{bak}  P. Bak, C. Tang, and K. Weisenfeld, Self-organized criticality: an explanation of $1/f$ noise, {\it Phys. Rev. Lett.} {\bf 59} (1987), 381--384.
 
\bibitem{bakbook} P. Bak,   {\it How Nature Works},   Oxford University Press, Oxford, 1997.     
 
 
 

 \bibitem{BLS}   A. Bj\"{o}rner, L. Lov\'{a}sz and P. Shor, Chip-firing games on graphs, {\it European J. Combin.} {\bf 12} (1991), 283-291.
 
 \bibitem{CGGMS}   S. Chapman, R. Garcia, L. Garc\'{i}a-Puente, M. Malandro, K.  Smith, Algebraic and combinatorial aspects of sandpile monoids on directed graphs,  {\it  J. Combin. Theory Ser. A} {\bf 120} (2013), 245--265.

\bibitem{Clark2017} L. O. Clark, D. Mart\'{\i}n Barquero, C. Mart\'{\i}n Gonz\'{a}lez,  M. Siles Molina, Using Steinberg algebras to study decomposability of Leavitt path algebras, {\it Forum Math.} {\bf 29} (2017), no. 6, 1311-1324.


\bibitem{D1}  D. Dhar, Self-organized critical state of sandpile automaton models, {\it Phys. Rev. Lett.} {\bf 64} (14) (1990), 1613--1616.




\bibitem{H}  R. Hazrat, The graded structure of Leavitt path algebras, {\it Israel J. Math.} {\bf 195} (2013), 833--895.

\bibitem{HN}  R. Hazrat and T. G. Nam, Unital algebras being Morita equivalent to weighted Leavitt path algebras, arXiv:2312.15704.

\bibitem{HR} R. Hazrat,  R. Preusser, Applications of normal forms for weighted Leavitt path algebras: simple rings and domains, {\it Algebr. Represent. Theory} {\bf 20} (2017), no. 5, 1061-1083. 

\bibitem{LamBook} T. Y. Lam, {\it A first course in noncommutative rings}, Second edition, Graduate Texts in Mathematics, 131. Springer-Verlag, New York, 2001. 

\bibitem{vitt62}W.G.  Leavitt, The module type of a ring, {\it Trans. Amer. Math. Soc.} {\bf 103} (1962), 113--130. 

\bibitem{LevinePeres} L. Levine and V. Peres, Laplacian growth, sandpiles, and scaling limits, {\it Bull. Amer. Math. Soc. (N.S.)} {\bf 54} (2017), no. 3, 355-382.




\bibitem{P}  R. Preusser, The $\mathcal{V}$-monoid of a weighted Leavitt path algebra, {\it Israel J. Math} {\bf 234} (2019), 125--147.

\bibitem{Pre}  R. Preusser, Weighted Leavitt path algebras, an overview, 	{\it Zap. Nauchn. Sem. S.-Peterburg. Otdel. Mat. Inst. Steklov. (POMI)} {\bf 531} (2024), 157-237.



\bibitem{T}  E. Toumpakari, \emph{On the abelian sandpile model}. Thesis (Ph.D.) -- The University of Chicago. 2005. 75 pp. ISBN: 978-0542-04489-2.

\end{thebibliography}
\end{document}